\documentclass[oneside,english]{amsart}
\usepackage{amsthm}
\usepackage{amstext}
\usepackage{amssymb}

\makeatletter
\numberwithin{equation}{section}
\numberwithin{figure}{section}
\theoremstyle{plain}
\newtheorem{thm}{Theorem}
  \theoremstyle{plain}
  \newtheorem{cor}[thm]{Corollary}
  \theoremstyle{remark}
  \newtheorem{rem}[thm]{Remark}
  \theoremstyle{definition}
  \newtheorem{defn}[thm]{Definition}
  \theoremstyle{plain}
  \newtheorem{lem}[thm]{Lemma}
  \theoremstyle{plain}
  \newtheorem{prop}[thm]{Proposition}

\DeclareMathOperator{\interior}{int}
\numberwithin{thm}{section}
\usepackage{graphicx}

\makeatother

\usepackage{babel}

\begin{document}

\title{Real Analytic Metrics on $S^{2}$ with Total Absence of Finite Blocking}

\author{Marlies Gerber and Lihuei Liu}
\begin{abstract}
If $(M,g)$ is a Riemannian manifold and $(x,y)\in M\times M,$ then
a set $P\subset M\setminus\{x,y\}$ is said to be a blocking set for
$(x,y)$ if every geodesic from $x$ to $y$ passes through a point
of $P.$ If no pair $(x,y)$ in $M\times M$ has a finite blocking
set, then $(M,g)$ is said to be totally insecure. We prove that there
exist real analytic metrics $h$ on $S^{2}$ such that $(S^{2},h)$
is totally insecure.
\end{abstract}



\email{gerber@indiana.edu, liul@umail.iu.edu}

\address{Department of Mathematics, Indiana University, Bloomington, IN 47405,
USA}

\address{Department of Mathematics, University of Tennessee, Knoxville, TN
37996, USA}

\maketitle

\section{Introduction}

Let $(M,g)$ be a $C^{\infty}$ Riemannian manifold. In this paper,
manifolds and surfaces are assumed to be without boundary. We consider
a geodesic on $M$ as a mapping $\gamma:I\to M,$ where $I$ is an
interval of positive length. The trace of $\gamma,$ denoted tr$(\gamma),$
is the image of $\gamma$ in $M.$ Unless specified otherwise, we
will assume that geodesics are parametrized by arc length. If $I=[a,b],$
$x=\gamma(a),$ and $y=\gamma(b),$ then $\gamma$ is said to be a
geodesic from $x$ to $y.$ Two geodesics $\gamma_{i}:I_{i}\to M,$
$i=1,2,$ will be considered the same if and only if $\gamma_{1}$
is equal to $\gamma_{2}$ composed with a translation that maps $I_{1}$
onto $I_{2}.$ A subset $P\subset M$ is called a \emph{blocking set}
for a collection of geodesics $\Gamma$ on $M,$ if tr$(\gamma)\cap P\ne\emptyset$
for every geodesic $\gamma$ in $\Gamma.$ If $\Gamma$ consists of
all geodesics from $x$ to $y$ and there exists a finite blocking
set $P\subset M\setminus\{x,y\},$ then $(x,y)$ is said to be \emph{secure};
otherwise $(x,y)$ is said to be \emph{insecure}. The Riemannian manifold
$(M,g)$ is defined to be \emph{secure} if every pair $(x,y)\in M$
is secure. If there exists an insecure pair $(x,y)\in M\times M,$
$(M,g)$ is said to be \emph{insecure}. Moreover, if every pair $(x,y)\in M\times M$
is insecure, then $(M,g)$ is defined to be \emph{totally insecure.}

E. Gutkin and V. Schroeder \cite{GutkinSchroeder} showed that flat
metrics are secure, and it is conjectured in \cite{BurnsGutkin,LafontSchmidt}
that these are the only secure metrics. On the other hand, there are
many examples of totally insecure metrics \cite{BangertGutkin,BurnsGutkin,GutkinSchroeder,LafontSchmidt},
and according to \cite{BurnsGutkin}, it is expected that ``most''
metrics are totally insecure. V. Bangert and Gutkin \cite{BangertGutkin}
proved that any compact Riemannian surface of genus greater than one
is totally insecure. In case $M$ is a compact surface of genus one,
they showed that there is a $C^{2}$ open and $C^{\infty}$ dense
collection of metrics on $M$ that are totally insecure. However,
in the case of surfaces of genus zero, there were no previously known
examples of totally insecure metrics. Gerber and W.-K. Ku \cite{GerberKu}
showed that on any compact Riemannian manifold $M$ of dimension greater
than one there is a dense G-delta set of $C^{\infty}$ metrics, $\mathcal{G},$
such that for each $g\in\mathcal{G},$ there is a dense G-delta subset
$\mathcal{R}=\mathcal{R}(g)$ of $M\times M$ consisting of insecure
pairs $(x,y)$ for $(M,g),$ but this result only provides ``generic''
insecurity, not total insecurity. The main theorem of our paper is
that if $M=S^{2}$ or $P^{2}(\mathbb{R}),$ then there exist real
analytic metrics $h$ on $M$ such that $(M,h)$ is totally insecure.
We present the argument for $M=S^{2},$ and the case $M=P^{2}(\mathbb{R})$
follows easily, as indicated in Remark \ref{rem:RP2}.

The real analytic metrics $h$ on $S^{2}$ for which we prove total
insecurity are obtained in the same way as the metrics on $S^{2}$
for which K. Burns and Gerber \cite{BurnsGerberI,BurnsGerberII} showed
that the geodesic flow is ergodic. Our proof relies on ideas in \cite{BangertGutkin}
and techniques in non-uniform hyperbolicity \cite{BarreiraPesin,KatokBurns}.
Burns and Gutkin \cite{BurnsGutkin} and, independently, J.-F. Lafont
and B. Schmidt \cite{LafontSchmidt} showed that compact Riemannian
manifolds (of any dimension) with no conjugate points whose geodesic
flows have positive topological entropy are totally insecure. In the
special case of a compact manifold with negative curvature the methods
in \cite{BangertGutkin} provide a different proof of total insecurity.
A key step in our paper is to show that there is a closed $h$-geodesic
$\rho$ such that for any pair $(x,y)\in S^{2}\times S^{2}$ there
is an infinite sequence of $h$-geodesics $(\gamma_{n})$ from $x$
to $y$ that stay arbitrarily close to $\rho$ except during a uniformly
bounded amount of time at the beginning and at the end of their parameter
intervals. (See Proposition \ref{pro:BangertGutkinProperty} for a
precise statement.) This condition is essentially taken from \cite{BangertGutkin},
and it is easy to establish if $(S^{2},h)$ is replaced by a manifold
of negative curvature and $\rho$ is replaced by any closed geodesic.
(For manifolds of negative curvature, the first statement in Proposition
\ref{pro:Asymptotic} is also easy to prove, for any closed geodesic
$\rho,$ but the analog of Proposition \ref{pro:BangertGutkinProperty}
can be proved directly.)

According to Proposition \ref{pro:Asymptotic}, for each $(x,y)\in S^{2}\times S^{2},$
there are $h$-geodesics $\gamma^{+}:[0,\infty)\to S^{2}$ and $\gamma^{-}:(-\infty,0]\to S^{2}$
with $\gamma^{+}(0)=x$ and $\gamma^{-}(0)=y$ such that $\gamma^{+}$
is asymptotic to the closed $h$-geodesic $\rho$ as $t\to\infty$
and $\gamma^{-}$ is asymptotic to $\rho$ as $t\to-\infty$ (as in
Definition \ref{def:DefAsymptotic}). The sequence of geodesics $(\gamma_{n})$
mentioned above is obtained by finding small perturbations $\gamma_{n}^{+}$
and $\gamma_{n}^{-}$ of $\gamma^{+}$ and $\gamma^{-},$ respectively,
and large positive $t_{n}^{+}$ and $t_{n}^{-}$ such that $\gamma_{n}^{+}|[0,t_{n}^{+}]$
and $\gamma_{n}^{-}|[-t_{n}^{-},0]$ can be smoothly joined at $\gamma_{n}^{+}(t_{n}^{+})=\gamma_{n}^{-}(-t_{n}^{-}).$

We show (in Proposition \ref{pro:AnalyticNonconcurrence}) that no
one-element set $\{z\}$ can be a blocking set for any infinite collection
of geodesics in $(S^{2},h)$ of uniformly bounded length starting
at a point $x.$ This is clearly true if $(S^{2},h)$ is replaced
by a manifold with no conjugate points. However, any Riemannian metric
on $S^{2}$ must have conjugate points (Remark 3.4 in Chapter 7 of
\cite{doCarmo}). In fact, for the metrics from \cite{BurnsGerberII},
there are conjugate points along any geodesic segment that passes
through a cap (as defined in Section \ref{sec:Construction-of-Metrics}
below). Our proof of Proposition \ref{pro:AnalyticNonconcurrence}
relies on the real analyticity of the metric.

The proof of Proposition \ref{pro:Asymptotic} utilizes the fact that
the geodesic flow for $(S^{2},h)$ is topologically transitive (which
follows from the ergodicity with respect to Liouville measure), but
ergodicity is not used in any other way.

\section{Construction of Totally Insecure Real Analytic Metrics on $S^{2}$\label{sec:Construction-of-Metrics}}

We describe the construction of real analytic metrics on compact surfaces
with ergodic geodesic flow, as in \cite{BurnsGerberII}. Let $S$
be a compact surface with a real analytic differentiable structure,
and let $g$ be a $C^{\infty}$ metric on $S$ that satisfies the
following conditions:
\begin{enumerate}
\item There is a finite (non-empty) disjoint collection of monotone curvature
caps $\mathcal{C}_{i},$ $i=1,\dots,q,$ (as defined below) such that
$S\setminus(\cup_{i=1}^{q}\mathcal{C}_{i})$ has negative curvature
with respect to $g.$
\item If $s$ is the signed distance from the boundary of a cap $\mathcal{C}=\mathcal{C}_{i},$
for $i\in\{1,\dots,q\},$ with $s>0$ in the interior of $\mathcal{C},$
then the curvature with respect to $g$ is $s$ for points in a neighborhood
of $\partial\mathcal{C}.$
\end{enumerate}
Note that assumption (2) implies that $g$ is real analytic in a neighborhood
of $\partial\mathcal{C}$. Metrics satisfying (1) and (2) exist on
every compact real analytic surface $S.$ (See Section 1 of \cite{BurnsGerberII}.)
Such metrics are viewed as having ``almost negative curvature,''
since many of the properties of geodesic flows on surfaces of negative
curvature extend to metrics of this type. They are of interest primarily
in the case of surfaces of genus zero or one, which do not support
metrics of negative curvature.

A \emph{cap} is defined to be a closed two-dimensional disk with nonnegative
curvature such that the boundary circle is the trace of a real analytic
geodesic. We say that a cap $\mathcal{C}$ has \emph{monotone curvature}
if it is radially symmetric and its curvature is a nondecreasing function
of the distance from the boundary of $\mathcal{C}.$

Real analytic metrics $h$ with ergodic geodesic flow are obtained
from the following two theorems in \cite{BurnsGerberII}. The proof
of Theorem \ref{thm:Cartan} is based on Cartan's theorem B \cite{Cartan}.
\begin{thm}
\label{thm:BurnsGerber} Let $S$ be a compact surface with a real
analytic differentiable structure, and let $g$ be a $C^{\infty}$
Riemannian metric on $S$ satisfying the above conditions (1) and
(2). Choose an open neighborhood $\Upsilon_{i}$ of $\partial\mathcal{C}_{i}$
for $i=1,\dots,q,$ and let $\Upsilon=\cup_{i=1}^{q}\Upsilon_{i}.$
Let $\mathbb{H}_{1}$ be the collection of $C^{3}$ Riemannian metrics
on $S$ that agree with $g$ to second order on $\partial\mathcal{C}_{i},$
$i=1,\dots,q.$ Then there exist a $C^{2}$ open neighborhood $\mathbb{H}_{2}$
of $g$ in the set of $C^{3}$ Riemannian metrics on $S,$ and a $C^{3}$
open neighborhood $\mathbb{H}_{3}$ of $g|\Upsilon$ in the set of
$C^{3}$ Riemannian metrics on $\Upsilon,$ such that the following
holds: If $h\in\mathbb{H}_{1}\cap\mathbb{H}_{2}$ and $h|\Upsilon\in\mathbb{H}_{3},$
then the geodesic flow for $h$ on $T^{1,h}S$ is ergodic with respect
to Liouville measure.
\begin{thm}
\label{thm:Cartan} Let $S$ be a compact surface with a real analytic
differentiable structure, and let $g$ be a $C^{\infty}$ Riemannian
metric on $S.$ Suppose that $\Gamma$ is a union of disjoint closed
real analytic curves on $S$ and there exists a neighborhood $\mathcal{U}$
of $\Gamma$ on which $g$ is real analytic. Then for any positive
integer $k$ there exists a real analytic metric $h$ on $S$ such
that $g$ and $h$ agree up to order $k$ on $\Gamma.$ Moreover,
$h$ can be taken arbitrarily close to $g$ in the $C^{\infty}$ topology.
\end{thm}
\end{thm}
The main result of this paper is the following:
\begin{thm}
\label{thm:GeneralSecurity} If $S,$ $g,$ $\Upsilon,$ and $\mathbb{H}_{1}$
are as in Theorem \ref{thm:BurnsGerber}, then there exist a $C^{2}$
open neighborhood $\mathbb{H}_{2}'$ of $g$ in the set of $C^{3}$
Riemannian metrics on $S,$ and a $C^{3}$ open neighborhood $\mathbb{H}_{3}'$
of $g|\Upsilon$ in the set of $C^{3}$ Riemannian metrics on $\Upsilon,$
such that the following holds: If $h\in\mathbb{H}_{1}\cap\mathbb{H}_{2}'$
and $h|\Upsilon\in\mathbb{H}_{3}',$ then $(S,h)$ is totally insecure.
\end{thm}
It follows from Theorem \ref{thm:Cartan} that the set of real analytic
metrics $h$ that satisfy the conclusions of Theorems \ref{thm:BurnsGerber}
and \ref{thm:GeneralSecurity} is nonempty. In particular, we obtain
the following corollary.
\begin{cor}
\label{cor:ExistenceAnalyticInsecure}If $S$ is a compact surface
with a real analytic differentiable structure, then there exist real
analytic metrics $h$ on $S$ such that $(S,h)$ is totally insecure.
\end{cor}
Theorem \ref{thm:GeneralSecurity} and Corollary \ref{cor:ExistenceAnalyticInsecure}
are only of interest in the case $S=S^{2}$ or $P^{2}(\mathbb{R}),$
since totally insecure metrics for positive genus surfaces were already
obtained in \cite{BangertGutkin}, as described in our introduction.
\begin{rem}
\label{rem:RP2} Suppose $P$ is a non-orientable compact real analytic
surface and $S$ is its orientable double cover, with covering map
$\pi:S\to P.$ It follows easily from the definition of total insecurity
that for any Riemannian metric $h_{P}$ on $P,$ $(P,h_{P})$ is totally
insecure if $(S,\pi^{*}h_{P})$ is totally insecure (see Proposition
1 in \cite{GutkinSchroeder}). This observation allows us to reduce
Theorem \ref{thm:GeneralSecurity} to the case in which the surface
is orientable.
\end{rem}
Throughout the rest of this paper we will assume that $S$ is a compact
orientable surface with a real analytic differentiable structure and
$g,$ $\Upsilon,$ $\mathbb{H}_{1},$ $\mathbb{H}_{2},$ and $\mathbb{H}_{3}$
are as in Theorem \ref{thm:BurnsGerber}. As observed above, the reader
may as well assume that $S=S^{2},$ although this does not matter
for the proof. Furthermore, we will assume that $h$ is a real analytic
metric with $h\in\mathbb{H}_{1}\cap\mathbb{H}_{2}'$ and $h|\Upsilon\in\mathbb{H}_{3}',$
where $\mathbb{H}_{2}'\subset\mathbb{H}_{2}$ is a $C^{2}$ open neighborhood
of $g$ in the set of $C^{3}$ Riemannian metrics on $S$, and $\mathbb{H}_{3}'\subset\mathbb{H}_{3}$
is a $C^{3}$ open neighborhood of $g|\Upsilon$ in the set of $C^{3}$
Riemannian metrics on $\Upsilon.$ Additional requirements on $\mathbb{H}_{2}'$
and $\mathbb{H}_{3}'$ will be imposed later in this section and in
Section \ref{sec:Cone Fields}.

Since $h$ agrees with $g$ to second order on $\partial\mathcal{C}_{i},$
for $i=1,\dots,q,$ each $\partial\mathcal{C}_{i}$ is the trace of
a closed geodesic for $h$, as it is for $g.$ We require $\mathbb{H}_{2}'$
and $\mathbb{H}_{3}'$ to be sufficiently small so that for $h\in\mathbb{H}_{1}\cap\mathbb{H}_{2}',$
and $h|\Upsilon\in\mathbb{H}_{3}',$ the curvature for $h$ is positive
in $\interior(\cup_{i=1}^{q}\mathcal{C}_{i})$ and negative in $\mathcal{N},$
which is defined by \begin{equation}
\mathcal{N}=S\setminus\left(\cup_{i=1}^{q}\mathcal{C}_{i}\right).\label{eq:NegativeCurvReg}\end{equation}

For our argument, we need a closed $h$-geodesic $\rho$ such that
for any $(x,y)\in S\times S,$ there is a family of $h$-geodesics
$(\gamma_{n})_{n=1,2,\dots}$ which accumulate near $\rho,$ as described
in Proposition \ref{pro:BangertGutkinProperty}. Any simple closed
geodesic along one of $\partial\mathcal{C}_{i},$ $i\in\{1,\dots,q\},$
and any closed geodesic whose trace lies in $\mathcal{N}$ could serve
as $\rho.$ In our argument, we choose to work with a closed geodesic
in $\mathcal{N},$ because the estimates required to prove the analog
of Proposition \ref{pro:Asymptotic} for a geodesic along one of $\partial\mathcal{C}_{i}$
are more difficult, due to the fact that the curvature of the surface
vanishes on $\partial\mathcal{C}_{i}.$ (Proposition \ref{pro:Asymptotic}
remains true for $\rho$ replaced by a geodesic along one of $\partial\mathcal{C}_{i}$,
except in the last line we would either have to replace $\tau_{2}(\widehat{\gamma}_{v_{+}})$
and $\tau_{2}(\widehat{\gamma}_{v_{-}})$ by $|\tau_{2}(\widehat{\gamma}_{v_{+}})|$
and $|\tau_{2}(\widehat{\gamma}_{v_{-}})|,$ respectively, or require
that the Fermi coordinates along $\partial\mathcal{C}_{i}$ be chosen
so that $\tau_{2}>0$ for points near $\partial\mathcal{C}_{i}$ that
lie in $\mathcal{N}.$ This modification is needed, because no geodesic
can be asymptotic to a geodesic along $\partial\mathcal{C}_{i}$ while
remaining in $\interior\mathcal{C}_{i}.$ No additional difficulties
in the proof of Theorem \ref{thm:GeneralSecurity} are caused by this
modification.)

We can find a closed $g$-geodesic $\rho_{0}$ in $\mathcal{N}$ by
applying the Birkhoff curve-shortening procedure \cite{Birkhoff,Croke}
(for the metric $g)$ to a closed curve $\alpha$ in $\mathcal{N}$
that is not homotopic within $\mbox{\ensuremath{\overline{\mathcal{N}}}}$
to a point or to any of the boundary components of $\overline{\mathcal{N}}.$
Since each boundary component of $\overline{\mathcal{N}}$ consists
of the trace of a simple closed geodesic, it follows that for any
two points $x,y\in\overline{\mathcal{N}}$ with dist$_{g}(x,y)<r,$
where $r$ is the injectivity radius of $(S,g),$ the length minimizing
$g$-geodesic $\gamma$ from $x$ to $y$ must have tr$(\gamma)\subset\overline{\mathcal{N}.}$
If we start the curve-shortening procedure by partitioning $\alpha$
into segments of length less than $r,$ then all of the curves obtained
from $\alpha$ with this procedure, as well as the limiting curve,
remain in $\overline{\mathcal{N}.}$ Moreover, all of these curves
are homotopic to $\alpha$ within $\overline{\mathcal{N}}.$ The limiting
curve is a closed $g$-geodesic $\rho_{0}$ in $\mathcal{N}.$ Since
the orbit of the geodesic glow for $g$ along $\rho_{0}$ is transversally
hyperbolic, it follows that if $\mathbb{H}_{2}'$ is a sufficiently
small $C^{2}$ neighborhood of $g,$ there is a closed $h$-geodesic
$\rho=\rho(h)$ that is $C^{3}$ close to $\rho$ (although we only
need it to be $C^{0}$ close to $\rho).$ The neighborhood $\mathbb{H}_{2}'$
can be chosen so that dist$_{h}\big(\text{tr}(\rho),\cup_{i=1}^{q}\mathcal{C}_{i}\big)$
is uniformly bounded away from $0$ for all $h\in\mathbb{H}_{2}'.$

It is possible that there is no \emph{simple} closed geodesic in $\mathcal{N}.$
For example, if $S$ is $S^{2}$ with three caps, then any simple
closed geodesic in $\mathcal{N}$ would be homotopic to the boundary
of one of the caps, which is impossible by the Gauss-Bonnet Theorem.

There is another type of closed geodesic that may occur, namely one
that passes through the interior of one or more of the caps. However,
the conjugate points that occur along such a geodesic prevent it from
being a suitable choice for $\rho$ in our argument.

Henceforth, unless otherwise specified, we will assume that geodesics,
geodesic flow, distances, lengths of vectors, curvature, etc., for
$S$ are taken with respect to $h.$ For $x\in S$ and $v\in T_{x}^{1}S,$
let $\gamma_{v}$ denote the geodesic with $\gamma_{v}(0)=x$ and
$\gamma_{v}'(0)=v.$ Let $\varphi^{t},$ $t\in\mathbb{R},$ be the
geodesic flow on $T^{1}S.$

\section{Stable and Unstable Cone Fields\label{sec:Cone Fields}}

We will define stable and unstable cone fields at each $v\in T^{1}S.$
These cone fields are essentially the same as those in \cite{BurnsGerberII},
except that the definitions are extended to be valid inside the caps.
The idea of extending the cone fields into the caps and obtaining
continuous stable and unstable line fields (as in our Lemma \ref{lem:ContinuousLineFields})
already appeared in V. Donnay's proof of the existence of $C^{\infty}$
metrics on $S^{2}$ with ergodic geodesic flow \cite{Donnay}, but
the real analytic case is different, because we no longer have invariance
of the $\mathcal{K}^{+}$ cones (see Definition \ref{def:K+Cones})
from the time that a geodesic enters a cap until it exits the cap.

The Riemannian metric $h=\langle\cdot,\cdot\rangle$ on $S$ induces
a Riemannian metric on $TS:$

\[
\langle\langle\xi,\eta\rangle\rangle=\langle\xi_{H},\eta_{H}\rangle+\langle\xi_{V},\eta_{V}\rangle,\]
where $H$ and $V$ denote the horizontal and vertical components,
respectively (see, e.g., Chapter 3, Exercise 2 in \cite{doCarmo}).
We will identify $\xi\in TTS$ with $(\xi_{H},\xi_{V}).$ If $x\in S,$
$w\in T_{x}^{1}S,$ and $\xi\in T_{w}T^{1}S,$ then $\langle\xi_{V},w\rangle=0.$
For $w\in T_{x}^{1}S,$ we let $\mathcal{P}(w)$ be the two-dimensional
subspace of $T_{w}T^{1}S$ defined by \begin{equation}
\mathcal{P}(w)=\{\xi\in T_{w}T^{1}S:\langle\xi_{H},w\rangle=0\}.\label{eq:Distribution}\end{equation}
 We define $H,V$ \emph{coordinates} on $\mathcal{P}(w)$ by choosing
$N\in T_{x}^{1}S$ such that $\langle N,w\rangle=0$ and letting $\xi=(\xi_{H},\xi_{V})\in\mathcal{P}(w)$
have coordinates $(\lambda_{1},\lambda_{2})$ if $\xi_{H}=\lambda_{1}N$
and $\xi_{V}=\lambda_{2}N.$ If $N$ is replaced by $-N,$ then the
coordinates change from $(\lambda_{1},\lambda_{2})$ to $(-\lambda_{1},-\lambda_{2}),$
but this does not matter for the cones and the lines through $0$
in $\mathcal{P}(w)$ that we consider below. The distribution $w\mapsto\mathcal{P}(w),$
$w\in T^{1}S,$ is orientable, because we may specify that the ordered
pair of vectors given in $H,V$ coordinates by $(1,0),(0,1)$ is positively
oriented. This orientation does not depend on the choice of $N.$

Suppose $w_{0}\in T_{x}^{1}S,$ $\xi\in T_{w_{0}}T^{1}S,$ and $w(s)\in T_{p(s)}^{1}S,$
$-s_{0}<s<s_{0},$ is a curve in $T^{1}S$ that is tangent to $\xi$
at $w_{0}=w(0)$ when $s=0.$ Then $J(t)=(d/ds)|_{s=0}\gamma_{w(s)}(t)$
is a Jacobi field along $\gamma_{w_{0}}$ with $J(0)=\xi_{H},$ $J'(0)=\xi_{V}$
and $((d\varphi^{t}(\xi))_{H},(d\varphi^{t}(\xi))_{V})=(J(t),J'(t)).$
In particular, if $\xi\in\mathcal{P}(w_{0}),$ then $\langle J(0),\gamma_{w_{0}}'(0)\rangle=0=\langle J'(0),\gamma_{w_{0}}'(0)\rangle,$
which implies that $\langle J(t),\gamma'_{w_{0}}(t)\rangle\equiv0.$
Thus $\langle(d\varphi^{t}(\xi))_{H},d\varphi^{t}(w_{0})\rangle\equiv0,$
i.e., the distribution $w\mapsto\mathcal{P}(w)$ is invariant under
the geodesic flow. Moreover, the orbits of the geodesic flow are orthogonal
to the distribution $w\mapsto\mathcal{P}(w),$ since $\eta_{H}=w_{0}$
and $\eta_{V}=0$ if $\eta=(D/dt)|_{t=0}\varphi^{t}(w_{0}).$

Now assume that the curve $w(s),$ as above, is a $C^{1}$ regular
curve in $T^{1}S$ that is everywhere tangent to the distribution
$\mathcal{P}$ (i.e., $w'(s)\in\mathcal{P}(w(s))$ for $-s_{0}<s<s_{0}).$
Then $w(s)$ is a unit normal field along the curve $p(s)$ in $S.$
(Throughout this paper, a \emph{regular} curve will mean a curve whose
derivative is nowhere vanishing.) The \emph{signed curvature $k(s)$
of $p(s)$ with respect to the unit normal field $w(s)$} is defined
by \begin{equation}
k(s)=-\frac{1}{|p'(s)|^{2}}\left\langle \frac{D(p'(s))}{ds},w(s)\right\rangle =\frac{1}{|p'(s)|^{2}}\left\langle p'(s),\frac{D(w(s))}{ds}\right\rangle .\label{eq:DefSignedCurvature}\end{equation}
Our choice of sign is such that unstable {[}stable{]} curves (to be
defined in Section \ref{sec:LyapunovLineFields}) that lie outside
the caps have positive {[}negative{]} curvature. The second equality
in (\ref{eq:DefSignedCurvature}) follows from the fact that $\langle p'(s),w(s)\rangle\equiv0.$

If $J(t)$ is a perpendicular Jacobi field along $\gamma_{w_{0}},$
we may write $J(t)=j(t)N(t),$ where $N(t)$ is a continuous unit
normal field along $\gamma_{w_{0}}(t),$ and $j(t)$ satisfies the
scalar Jacobi equation \[
j''(t)+K(\gamma_{w_{0}}(t))j(t)=0,\]
 where $K$ is the Gaussian curvature. If $w(s),$ $-s_{0}<s<s_{0},$
with $w(0)=w_{0},$ is a $C^{1}$ regular curve in $T^{1}S$ that
is everywhere tangent to $\mathcal{P}$ and $J(t)=(d/ds)\vert_{s=0}\gamma_{w(s)}(t)=j(t)N(t),$
then\[
J'(t)=(D/dt)(d/ds)|_{s=0}\gamma_{w(s)}(t)=(D/ds)|_{s=0}(d\varphi^{t}(w(s))),\]
 and $j'(t)/j(t)=\langle J(t),J'(t)\rangle/|J(t)|^{2}.$ From the
second version of the formula for $k(s)$ in (\ref{eq:DefSignedCurvature}),
it follows that $j'(t)/j(t)$ is equal to the signed curvature at
$s=0$ of the curve $s\mapsto\gamma_{w(s)}(t)$ with respect to the
unit normal field $\varphi^{t}(w(s)).$ If $m(t)=j'(t)/j(t),$ the
slope of $d\varphi^{t}(w'(0))$ in the $H,V$ coordinate system, then
$m(t)$ satisfies the Riccati equation\begin{equation}
m^{2}(t)+m'(t)+K(\gamma_{w_{0}}(t))=0.\label{eq:Riccati}\end{equation}
 The Riccati equation can be transformed by setting $\theta=\tan^{-1}(m),$
to obtain

\begin{equation}
\theta'(t)+\sin^{2}(\theta(t))+K(\gamma_{w_{0}}(t))\cos^{2}(\theta(t))=0.\label{eq:TransformedRiccati}\end{equation}
 Here $\theta\in\mathbb{R}/\pi\mathbb{Z},$ which we identify with
$(-\pi/2,\pi/2].$

\begin{figure}[htbp]\begin{center}
\begin{picture}(0,0)%
\includegraphics[scale=0.7]{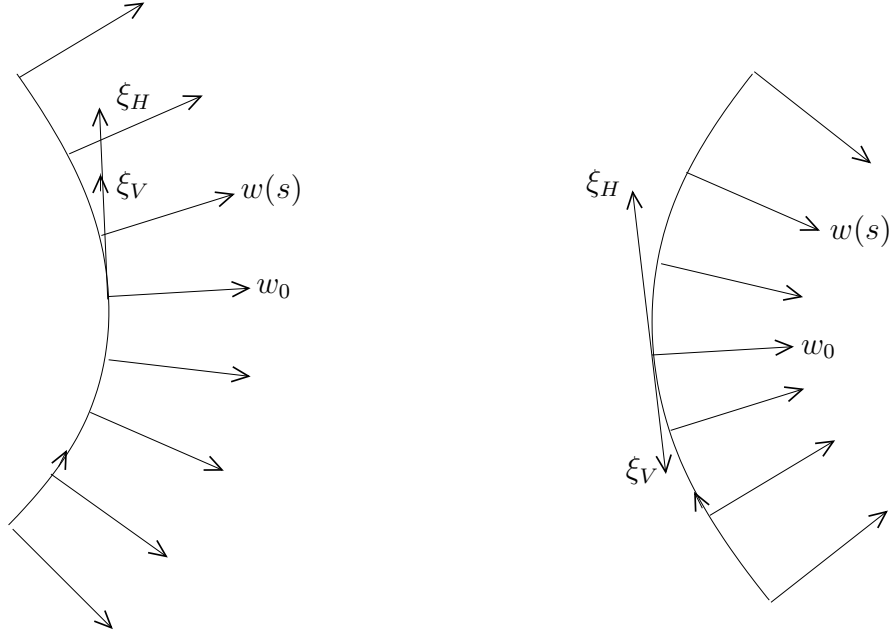}%
\end{picture}%
\setlength{\unitlength}{3947sp}%
\begingroup\makeatletter\ifx\SetFigFont\undefined%
\gdef\SetFigFont#1#2#3#4#5{%
  \reset@font\fontsize{#1}{#2pt}%
  \fontfamily{#3}\fontseries{#4}\fontshape{#5}%
  \selectfont}%
\fi\endgroup%
\begin{picture}(7915,3750)(2014,-8360) \put(2700,-5611){\makebox(0,0)[lb]{\smash{{\SetFigFont{12}{14.4}{\familydefault}{\mddefault}{\updefault}{$\xi_V$}%
}}}} \put(3480,-5640){\makebox(0,0)[lb]{\smash{{\SetFigFont{12}{14.4}{\familydefault}{\mddefault}{\updefault}{$w(s)$}%
}}}} \put(7180,-5900){\makebox(0,0)[lb]{\smash{{\SetFigFont{12}{14.4}{\familydefault}{\mddefault}{\updefault}{$w(s)$}%
}}}} \put(3580,-6240){\makebox(0,0)[lb]{\smash{{\SetFigFont{12}{14.4}{\familydefault}{\mddefault}{\updefault}{$w_0$}%
}}}} \put(7000,-6620){\makebox(0,0)[lb]{\smash{{\SetFigFont{12}{14.4}{\familydefault}{\mddefault}{\updefault}{$w_0$}%
}}}} \put(2700,-5050){\makebox(0,0)[lb]{\smash{{\SetFigFont{12}{14.4}{\familydefault}{\mddefault}{\updefault}{$\xi_H$}%
}}}} \put(5650,-5600){\makebox(0,0)[lb]{\smash{{\SetFigFont{12}{14.4}{\familydefault}{\mddefault}{\updefault}{$\xi_H$}%
}}}} \put(5900,-7400){\makebox(0,0)[lb]{\smash{{\SetFigFont{12}{14.4}{\familydefault}{\mddefault}{\updefault}{$\xi_V$}%
}}}}

\end{picture}%
\caption{The tangent vector $\xi=(\xi_H,\xi_V)$ at $w_0$  to the curve $w(s)$  in $T^1S$ is in ${\mathcal{K}}^+_{w_0}$ for the curve on the left and is in ${\mathcal{K}}^-_{w_0}$ for the curve on the right.}  \label{figure:3} \end{center} \end{figure}
\begin{defn}
\label{def:K+Cones} For $w_{0}\in T_{x}^{1}S,$ we define cones $\mathcal{K}_{w_{0}}^{+},\mathcal{K}_{w_{0}}^{-}\subset\mathcal{P}(w_{0})$
by \[
\mathcal{K}_{w_{0}}^{+}=\{\xi\in\mathcal{P}(w_{0}):\langle\xi_{H},\xi_{V}\rangle\ge0\}\text{\ and\ }\mathcal{K}_{w_{0}}^{-}=\{\xi\in\mathcal{P}(w_{0}):\langle\xi_{H},\xi_{V}\rangle\le0\}.\]
The $\mathcal{K}_{w_{0}}^{+}$ cones correspond to perpendicular Jacobi
fields with $jj'\ge0.$ If the curvature is nonpositive along $\gamma_{w_{0}}(t)$
for $t_{0}\le t\le t_{1},$ then $(j(t)j'(t))'=(j'(t))^{2}-K(\gamma_{w_{0}}(t))(j(t))^{2}\ge0$
for $t_{0}\le t\le t_{1}.$ This implies that \begin{equation}
d\varphi^{t_{1}-t_{0}}\mathcal{K}_{\gamma'_{w_{0}}(t_{0})}^{+}\subset\mathcal{K}_{\gamma'_{w_{0}}(t_{1})}^{+}\text{\ and\ }d\varphi^{t_{0}-t_{1}}\mathcal{K}_{\gamma'_{w_{0}}(t_{1})}^{-}\subset\mathcal{K}_{\gamma'_{w_{0}}(t_{0})}^{-}.\label{eq:NegCurvConeInvariance}\end{equation}
If, in addition, the curvature is negative at a point $\gamma_{w_{0}}(t),$
for some $t\in[t_{0},t_{1}],$ then \begin{equation}
d\varphi^{t_{1}-t_{0}}\mathcal{K}_{\gamma'_{w_{0}}(t_{0})}^{+}\subset\interior\mathcal{K}_{\gamma'_{w_{0}}(t_{1})}^{+}\text{\ and\ }d\varphi^{t_{0}-t_{1}}\mathcal{K}_{\gamma'_{w_{0}}(t_{1})}^{-}\subset\interior\mathcal{K}_{\gamma'_{w_{0}}(t_{0})}^{-},\label{eq:StrictNegConeInv}\end{equation}
where $\interior\mathcal{K},$ for a cone $\mathcal{K}\subset\mathcal{P}(w_{1}),$
means the topological interior of $\mathcal{K}$ within $\mathcal{P}(w_{1})$
together with $0\in T_{w_{1}}T^{1}S.$

For each cap $\mathcal{C}_{i},$ $i=1,\dots,q,$ we choose closed
disks $\mathcal{D}_{i}$ and $\mathcal{E}_{i}$ in $S$ that are radially
symmetric about the center of $\mathcal{C}_{i}$ for the $C^{\infty}$
metric $g$ such that $\mathcal{C}_{i}\subset\interior\mathcal{D}_{i}$
and $\mathcal{D}_{i}\subset\interior\mathcal{E}_{i}.$ We require
that $\mathcal{E}_{i}\cap\mathcal{E}_{j}=\emptyset$ if $i\ne j,$
and that the closed $h$-geodesic $\rho$ constructed in Section \ref{sec:Construction-of-Metrics}
lie in $S\setminus(\cup_{i=1}^{q}\mathcal{E}_{i}).$ Since the curvature
for $g$ is negative in $\mathcal{E}_{i}\setminus\mathcal{C}_{i},$
dist$_{g}(\gamma(t),\partial\mathcal{C}_{i})$ is a strictly convex
function of $t$ for any $g$-geodesic $\gamma$ in $\mathcal{E}_{i}\setminus\mathcal{C}_{i}.$
This implies that there exists $\tilde{\beta}_{i}>0$ such that for
any $g$-geodesic $\gamma$ with $\gamma(0)\in\partial\mathcal{D}_{i}$
and $\gamma'(0)$ either tangent to $\partial\mathcal{D}_{i}$ or
$\gamma'(0)$ pointing strictly out of $\mathcal{D}_{i}$ (i.e., $\gamma'(0)$
and $\mathcal{D}_{i}$ are on opposite sides of the tangent line to
$\partial\mathcal{D}_{i}$ at $\gamma(0))$, we have $\gamma((0,\tilde{\beta}_{i}))\subset S\setminus\mathcal{D}_{i}.$
We will assume that $\mathbb{H}_{2}'$ is sufficiently small (i.e.,
$h$ is sufficiently close to $g$ in the $C^{2}$ topology) such
that the analogous property holds with $g$-geodesics replaced by
$h$-geodesics and $\tilde{\beta}_{i}$ replaced by some $\beta_{i}>0.$
We will refer to this property (for the $h$-metric) as the \emph{strong
convexity} of $\mathcal{D}_{i}.$
\end{defn}
If $x\in\interior\mathcal{C}_{i}$ and $v\in T_{x}^{1}S,$ then there
exist $a,b,\widehat{a},\widehat{b}$ with $a<\widehat{a}<0<\widehat{b}<b,$
such that $\gamma_{v}(t)\in\interior\mathcal{C}_{i}$ for $t\in(\widehat{a},\widehat{b}),$
$\gamma_{v}(t)\in\interior(\mathcal{D}_{i}\setminus\mathcal{C}_{i})$
for $t\in(a,\widehat{a})\cup(\widehat{b},b),$ and $\gamma_{v}(a),$
$\gamma_{v}(b)\in\partial\mathcal{D}_{i}.$ That is, $\gamma_{v}$
exits $\mathcal{C}_{i}$ in both positive and negative time, and once
it exits $\mathcal{C}_{i}$ (in either positive or negative time)
it exits $\mathcal{D}_{i}$ without first re-entering $\mathcal{C}_{i}.$
This follows from Propositions 2.4 and 4.3 in \cite{BurnsGerberII},
provided that $g$ and $h$ are sufficiently close in the $C^{2}$
topology, and $g|\Upsilon_{i}$ and $h|\Upsilon_{i}$ are sufficiently
close in the $C^{3}$ topology, where $\Upsilon_{i}$ is as in Theorem
\ref{thm:BurnsGerber}.
\begin{lem}
\label{lem:ImprovedBurnsGerber}Let $\mathcal{C}=\mathcal{C}_{i}$
and $\mathcal{D}=\mathcal{D}_{i}$ for some $i\in\{1,\dots,q\}$ be
as above. Let $\Upsilon$ be as in Theorem \ref{thm:BurnsGerber}.
If $\mathbb{H}_{2}'$ and $\mathbb{H}_{3}'$ are sufficiently small
(i.e., $g$ and $h$ are sufficiently close in the $C^{2}$ topology,
and $g|\Upsilon$ and $h|\Upsilon$ are sufficiently close in the
$C^{3}$ topology), then for any geodesic $\gamma$ in $(\mathcal{D},h)$
with $\gamma(t)\in\interior\mathcal{D}$ for $t\in(a,b)$ and $\gamma(a),\gamma(b)\in\partial\mathcal{D},$
we have

\begin{equation}
d\varphi_{\gamma'(a)}^{b-a}(\mathcal{K}_{\gamma'(a)}^{+})\subset\mathcal{K}_{\gamma'(b)}^{+}.\label{eq:CapConeInvariance}\end{equation}
 Moreover, if $\gamma(t)\in\interior\mathcal{C}$ for $t\in(\widehat{a},\widehat{b})$
and $\gamma(\widehat{a}),\gamma(\widehat{b})\in\partial\mathcal{C},$
where $a<\widehat{a}<\widehat{b}<b,$ then \begin{equation}
d\varphi_{\gamma'(a)}^{\widehat{b}-a}(\mathcal{K}_{\gamma'(a)}^{+})\subset\mathcal{K}_{\gamma'(\widehat{b})}^{+}.\label{eq:StrongCapConeInvariance}\end{equation}
\end{lem}
\begin{proof}
If tr$(\gamma)$ does not intersect $\mathcal{C},$ then (\ref{eq:CapConeInvariance})
follows from (\ref{eq:NegCurvConeInvariance}), because the curvature
is negative in $\mathcal{D}\setminus\mathcal{C}.$ The idea for the
proof of the $\mathcal{K}^{+}$ invariance property in (\ref{eq:CapConeInvariance})
in the case that tr$(\gamma)$ intersects $\mathcal{C}$ is that the
$h$-geodesic $\gamma$ can be approximated by a $g$-geodesic that
also passes through $\mathcal{C}.$ From Proposition 2.7 in \cite{BurnsGerberII}
we know that from the time that the corresponding $g$-geodesic enters
$\mathcal{C}$ to the time that it exits $\mathcal{C}$ we have invariance
of the $\mathcal{K}^{+}$ cones under the derivative of the geodesic
flow for $g.$ This invariance property can be destroyed when the
metric $g$ is replaced by the metric $h.$ However, following $\gamma$
for additional time before and after it enters $\mathcal{C},$ while
it is in the negative curvature region $\mathcal{D}\setminus\mathcal{C},$
allows us to recover the $\mathcal{K}^{+}$ invariance in (\ref{eq:CapConeInvariance}).
This is proved in detail in Section 4 of \cite{BurnsGerberII}. (See
Proposition 4.10 of \cite{BurnsGerberII}.) Moreover, the estimates
in Propositions 4.8, 4.9, and 4.10 in \cite{BurnsGerberII} show that
it is actually enough to follow $\gamma$ just for additional time
${\it before}$ it enters $\mathcal{C},$ which leads to the containment
in (\ref{eq:StrongCapConeInvariance}). \end{proof}
\begin{defn}
\label{def:DefAsymptotic} A geodesic $\gamma_{1}:[a_{1},\infty)\to S$
is said to be \emph{asymptotic as} $t\to\infty$ \emph{to} a closed
geodesic $\gamma_{2}:[a_{2},b_{2}]\to S$ if there exists $t_{0}\in\mathbb{R}$
such that, after extending the domain of $\gamma_{2}$ to $(-\infty,\infty),$
we have \begin{equation}
\lim_{t\to\infty}\text{dist}(\gamma_{1}(t),\gamma_{2}(t_{0}+t))=0.\label{eq:AsymptoticEquation}\end{equation}
Similarly, a geodesic $\gamma_{1}:(-\infty,a_{1}]\to S$ is said to
be \emph{asymptotic as} $t\to-\infty$ \emph{to} a closed geodesic
$\gamma_{2}$ if there exists $t_{0}\in\mathbb{R}$ such that (\ref{eq:AsymptoticEquation})
holds with ``$\lim_{t\to\infty}$'' replaced by ``$\lim_{t\to-\infty}$''.
\end{defn}
From the usual procedure for constructing horocycles in regions of
nonpositive curvature (see, e.g.,\cite{HeintzeImHof}), we know that
for each $x\in\mathcal{D}_{i}\setminus\mathcal{C}_{i}$ there are
exactly two vectors $v_{x,j}\in T_{x}^{1}S,$ $j=1,2,$ corresponding
to the two possible orientations on $\partial\mathcal{C}_{i},$ such
that $\gamma_{v_{x,j}}(t)\in\interior(\mathcal{D}_{i}\setminus\mathcal{C}_{i})$
for all $t<0,$ and $\gamma_{v_{x,j}}$ is asymptotic to a closed
geodesic along $\partial\mathcal{C}_{i}$ as $t\to-\infty.$ If $x\in\interior(\mathcal{D}_{i}\setminus\mathcal{C}_{i})$
and $w\in T_{x}^{1}S,$ $w\ne v_{x,j},$ $j=1,2,$ then $\gamma_{w}(t)$
exits $\mathcal{D}_{i}$ in negative time, and one of the following
must occur:
\begin{enumerate}
\item There exists $a<0$ such that $\gamma_{w}(t)\in\interior(\mathcal{D}_{i}\setminus\mathcal{C}_{i})$
for $t\in(a,0]$ and $\gamma_{w}(a)\in\partial\mathcal{D}_{i};$ or
\item There exist $a<c<d<0$ such that $\gamma_{w}(t)\in\interior(\mathcal{D}_{i}\setminus\mathcal{C}_{i})$
for $t\in(a,c)\cup(d,0],$ $\gamma_{w}(t)\in\interior\mathcal{C}_{i}$
for $t\in(c,d),$ and $\gamma_{w}(a)\in\partial\mathcal{D}_{i}.$
\end{enumerate}
If $x\in\mathcal{C}_{i}$ and $w\in T_{x}^{1}S$ are such that $w$
is not tangent to $\partial\mathcal{C}_{i},$ then again $\gamma_{w}(t)$
exits $\mathcal{D}_{i}$ in negative time and we have:
\begin{enumerate}
\item [(3)]There exist $a<c\le0$ such that $\gamma_{w}(t)\in\interior(\mathcal{C}_{i})$
if $c<t<0,$ $\gamma_{w}(t)\in\interior(\mathcal{D}_{i}\setminus\mathcal{C}_{i})$
for $t\in(a,c),$ and $\gamma_{w}(a)\in\partial\mathcal{D}_{i}.$\end{enumerate}
\begin{defn}
\label{def:UnstableConeFields}If $x\in\interior\mathcal{D}$ for
$\mathcal{D}=\mathcal{D}_{i}$, $i\in\{1,\dots,q\},$ and $w\in T_{x}^{1}S$
is such that there exists $a<0$ with $\gamma_{w}(t)\in\interior\mathcal{D}$
for $t\in(a,0]$ and $\gamma_{w}(a)\in\partial\mathcal{D},$ we define
the \emph{unstable cone} $\mathcal{K}_{w}^{u}$ by

\[
\mathcal{K}_{w}^{u}=d\varphi^{-a}\mathcal{K}_{\gamma_{w}(a)}^{+}.\]
For all other $w\in T^{1}S,$ we define

\[
\mathcal{K}_{w}^{u}=\mathcal{K}_{w}^{+}.\]
 Similarly, if $x\in\interior\mathcal{D}$ for $\mathcal{D}=\mathcal{D}_{i}$,
$i\in\{1,\dots,q\},$ and $w\in T_{x}^{1}S$ is such that there exists
$b>0$ with $\gamma_{w}(t)\in\interior\mathcal{D}$ for $t\in[0,b)$
and $\gamma_{w}(b)\in\partial\mathcal{D},$ we define the \emph{stable
cone} $\mathcal{K}_{w}^{s}$ by

\[
\mathcal{K}_{w}^{s}=d\varphi^{-b}\mathcal{K}_{\gamma_{w}(b)}^{-}.\]
For all other $w\in T^{1}S,$ we define

\[
\mathcal{K}_{w}^{s}=\mathcal{K}_{w}^{-}.\]

\end{defn}
With these definitions, it follows from (\ref{eq:NegCurvConeInvariance})
and (\ref{eq:CapConeInvariance}) that the unstable {[}stable{]} cones
are invariant for $d\varphi^{t},$ $t\ge0$ {[}$t\le0${]}. That is,
\begin{equation}
d\varphi^{t}\mathcal{K}_{w}^{u}\subset\mathcal{K}_{\varphi^{t}w}^{u},\text{for\ }t\ge0,\label{eq:UnstableConeInvariance}\end{equation}
 and

\begin{equation}
d\varphi^{t}\mathcal{K}_{w}^{s}\subset\mathcal{K}_{\varphi^{t}w}^{s},\text{for\ }t\le0.\label{eq:StableConeInvariance}\end{equation}
 Moreover, by (\ref{eq:StrictNegConeInv}), if the basepoint of $\varphi^{\overline{t}}w,$
for some $\overline{t}$ between $0$ and $t,$ lies in $\mathcal{N}_{0},$
defined by

\begin{equation}
\mathcal{N}_{0}=S\setminus\left(\cup_{i=1}^{q}\mathcal{D}_{i}\right),\label{eq:N_0Def}\end{equation}
 then we have

\begin{equation}
d\varphi^{t}\mathcal{K}_{w}^{u}\subset\interior\mathcal{K}_{\varphi^{t}w}^{u},\text{\ if \ }t>0\label{eq:StrictUnstableInvariance}\end{equation}
 and \begin{equation}
d\varphi^{t}\mathcal{K}_{w}^{s}\subset\interior\mathcal{K}_{\varphi^{t}w}^{s},\text{\ if\ }t<0.\label{eq:StrictStableInvariance}\end{equation}

\begin{figure}[htbp]\begin{center}
\begin{picture}(0,0)%
\includegraphics[scale=0.5]{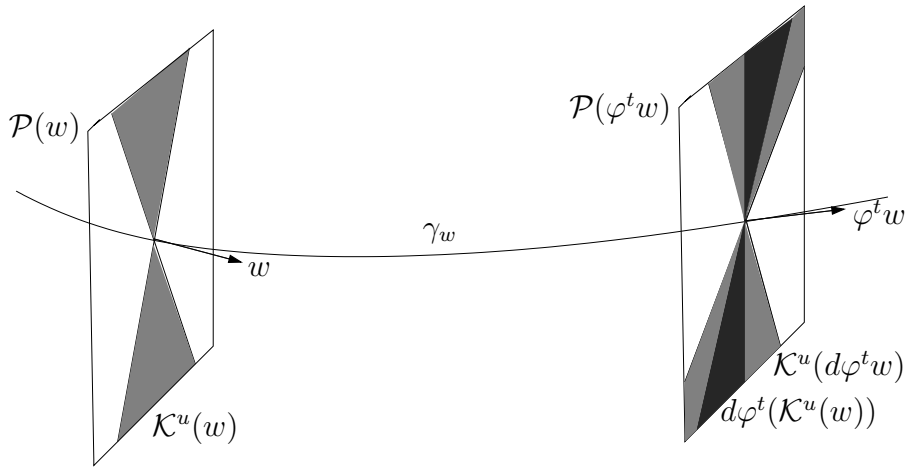}%
\end{picture}%
\setlength{\unitlength}{3947sp}%
\begingroup\makeatletter\ifx\SetFigFont\undefined%
\gdef\SetFigFont#1#2#3#4#5{%
  \reset@font\fontsize{#1}{#2pt}%
  \fontfamily{#3}\fontseries{#4}\fontshape{#5}%
  \selectfont}%
\fi\endgroup%
\begin{picture}(10974,2799)(439,-6073) \put(1900,-4880){\makebox(0,0)[lb]{\smash{{\SetFigFont{12}{14.4}{\familydefault}{\mddefault}{\updefault}{$w$}%
}}}} \put(5700,-4550){\makebox(0,0)[lb]{\smash{{\SetFigFont{12}{14.4}{\familydefault}{\mddefault}{\updefault}{$\varphi^tw$}%
}}}} \put(1300,-5880){\makebox(0,0)[lb]{\smash{{\SetFigFont{12}{14.4}{\familydefault}{\mddefault}{\updefault}{${\mathcal{K}}^u(w)$}%
}}}} \put(5200,-5500){\makebox(0,0)[lb]{\smash{{\SetFigFont{12}{14.4}{\familydefault}{\mddefault}{\updefault}{${\mathcal{K}}^u(d\varphi^tw)$}%
}}}} \put(4870,-5790){\makebox(0,0)[lb]{\smash{{\SetFigFont{12}{14.4}{\familydefault}{\mddefault}{\updefault}{$d\varphi^t({\mathcal{K}}^u(w))$}%
}}}} \put(3000,-4630){\makebox(0,0)[lb]{\smash{{\SetFigFont{12}{14.4}{\familydefault}{\mddefault}{\updefault}{$\gamma_w$}%
}}}} \put(400,-4000){\makebox(0,0)[lb]{\smash{{\SetFigFont{12}{14.4}{\familydefault}{\mddefault}{\updefault}{${\mathcal{P}}(w)$}%
}}}} \put(3920,-3880){\makebox(0,0)[lb]{\smash{{\SetFigFont{12}{14.4}{\familydefault}{\mddefault}{\updefault}{${\mathcal{P}}(\varphi^tw)$}%
}}}} \end{picture}%
\caption{Invariance of ${\mathcal{K}}^u$ cones under $d\varphi^t$}  \label{figure:3} \end{center} \end{figure}The following lemma will be used to prove a transversality condition
that is needed in Lemma \ref{lem:StableConnecting}.
\begin{lem}
\label{lem:DisjointCones}For all $w\in T^{1}S,$ we have

\begin{equation}
\interior\mathcal{K}_{w}^{u}\cap\interior\mathcal{K}_{w}^{s}=\{0\}.\label{eq:DisjointCones}\end{equation}
\end{lem}
\begin{proof}
If $w\in T_{x}^{1}S$ for $x\in S\setminus(\cup_{i=1}^{q}\interior\mathcal{D}_{i}),$
then (\ref{eq:DisjointCones}) is clear, because $\mathcal{K}_{w}^{u}=\mathcal{K}_{w}^{+}$
and $\mathcal{K}_{w}^{s}=\mathcal{K}_{w}^{-}.$ If $x\in\cup_{i=1}^{q}\interior\mathcal{D}_{i},$
(\ref{eq:DisjointCones}) follows from the definition of the stable
and unstable cones and (\ref{eq:CapConeInvariance}) in Lemma \ref{lem:ImprovedBurnsGerber}.
\end{proof}
The unstable cone field is continuous at those $v\in T_{x}^{1}S$
where one of the following holds:
\begin{enumerate}
\item $x\in(\cup_{i=1}^{q}\interior\mathcal{C}_{i})\cup(S\setminus(\cup_{i=1}^{q}\mathcal{D}_{i})).$
\item For some $i\in\{1,\dots,q\},$ $x\in\partial\mathcal{C}_{i}$ and
$v$ is not tangent to $\partial\mathcal{C}_{i}.$
\item For some $i\in\{1,\dots,q\},$ $x\in\interior(\mathcal{D}_{i}\setminus\mathcal{C}_{i})$
and $v\ne v_{x,1},v_{x,2},$ where $v_{x,1},v_{x,2}$ are as defined
above.
\item For some $i\in\{1,\dots,q\},$ $x\in\partial\mathcal{D}_{i}$ and
$v$ points strictly into $\mathcal{D}_{i}$ (that is, $v$ is not
tangent to $\partial\mathcal{D}_{i},$ and $v$ and $\mathcal{D}_{i}$
lie on the same side of the tangent line to $\partial\mathcal{D}_{i}$
at $x$).
\end{enumerate}
Analogous conditions can be given that guarantee that the stable cone
field is continuous at certain vectors $v\in T_{x}^{1}S.$
\begin{defn}
If $v\mapsto\mathcal{K}_{v}\subset\mathcal{P}(v)$ is a cone field
defined for $v$ in a neighborhood of $v_{0}$ in $T^{1}S$ such that
each $\mathcal{K}_{v}$ is closed, then we say that this cone field
is \emph{upper semi-continuous} at $v_{0}$ if the following holds:
for any sequence of vectors $(v_{n})_{n=1,2,\dots}$ in this neighborhood
of $v_{0}$ and a corresponding sequence $(\xi_{n})_{n=1,2,\dots}$
with $\xi_{n}\in\mathcal{K}_{v_{n}}$ such that $\lim_{n\to\infty}v_{n}=v_{0}$
and $\lim_{n\to\infty}\xi_{n}=\xi_{0},$ we must have $\xi_{0}\in\mathcal{K}_{v_{0}}.$ \end{defn}
\begin{lem}
\label{lem:UpperSemiContinuity}The unstable and stable cone fields
$v\mapsto\mathcal{K}_{v}^{u}$ and $v\mapsto\mathcal{K}_{v}^{s}$
given in Definition \ref{def:UnstableConeFields} are upper semi-continuous
on $(T^{1}S)\setminus(\cup_{i=1}^{q}T^{1}(\partial\mathcal{C}_{i})).$ \end{lem}
\begin{proof}
We will prove this for the unstable cone field. The proof for the
stable cone field is similar. Let $v_{0}\in T_{x}^{1}S$ be such that
$v_{0}$ is not tangent to any of $\partial\mathcal{C}_{i},$ $i=1,\dots,q.$
The following two cases are not covered by the above cases (1)-(4),
at which we have continuity:
\begin{enumerate}
\item [(i)]For some $i\in\{1,\dots,q\},$ $x\in\partial\mathcal{D}_{i}$
and $v_{0}$ is either tangent to $\partial\mathcal{D}_{i}$ or points
strictly out of $\mathcal{D}_{i}.$
\item [(ii)]For some $i\in\{1,\dots,q\},$ $x\in\interior(\mathcal{D}_{i}\setminus\mathcal{C}_{i})$
and $v_{0}\in\{v_{x,1},v_{x,2}\}.$
\end{enumerate}
${\it \ \ \ Case\ (i).}$ In this case, $\mathcal{K}_{v_{0}}^{u}=\mathcal{K}_{v_{0}}^{+}.$
For $v$ close to $v_{0}$ with basepoint in $S\setminus\interior\mathcal{D}_{i},$
or with basepoint $y\in\interior\mathcal{D}_{i}$ and $v\in\{v_{y,1},v_{y,2}\},$
we have $\mathcal{K}_{v}^{u}=\mathcal{K}_{v}^{+}.$ Thus it suffices
to consider $v$ close to $v_{0}$ with basepoint in $\interior\mathcal{D}_{i}$
and $\gamma_{v}$ ${\it not}$ asymptotic to $\partial\mathcal{C}_{i}$
as $t\to-\infty.$ Let $a=a(v),$ $b=b(v),$ $a<0,$ be such that
$\gamma_{v}((a,b))\subset\interior\mathcal{D}_{i}$ and $\gamma_{v}(a),$
$\gamma_{v}(b)\in\partial\mathcal{D}_{i}.$ If $v$ is close to $v_{0},$
then $b$ is close to $0.$ Thus $\mathcal{K}_{v}^{u}=d\varphi_{\gamma'(a)}^{-a}\mathcal{K}_{\gamma'(a)}^{+}$
is close to $\mathcal{K}_{\gamma'(b)}^{u}=d\varphi_{\gamma'(a)}^{b-a}\mathcal{K}_{\gamma'(a)}^{+},$
which is contained in $\mathcal{K}_{\gamma'(b)}^{+},$ by (\ref{eq:CapConeInvariance})
in Lemma \ref{lem:ImprovedBurnsGerber}. This establishes upper semi-continuity
at $v_{0}.$

${\it Case\ (ii).}$ Again we have $\mathcal{K}_{v_{0}}^{u}=\mathcal{K}_{v_{0}}^{+}.$
Moreover, $\gamma_{v_{0}}$ is asymptotic to a closed geodesic along
$\partial\mathcal{C}_{i}$ as $t\to-\infty.$ For $v$ close to $v_{0}$
such that $\gamma_{v}$ is also asymptotic to $\partial\mathcal{C}_{i}$
as $t\to-\infty,$ we have $\mathcal{K}_{v}^{u}=\mathcal{K}_{v}^{+}.$
There are two other possibilities: If $v$ is close to $v_{0}$ and
$\gamma_{v}$ exits $\mathcal{D}_{i}$ in negative time without first
entering $\mathcal{C}_{i},$ then $\mathcal{K}_{v}^{u}\subset\mathcal{K}_{v}^{+}$
by (\ref{eq:NegCurvConeInvariance}). If $v$ is close to $v_{0}$
and $\gamma_{v}$ passes through $\mathcal{C}_{i}$ in negative time
before exiting $\mathcal{D}_{i},$ let $a<\widehat{a}<\widehat{b}<0$
be such that $\gamma_{v}([a,0))\subset\mathcal{D}_{i},$ $\gamma_{v}([\widehat{a},\widehat{b}])\subset\mathcal{C}_{i},$
$\gamma_{v}(a)\in\partial\mathcal{D}_{i},$ and $\gamma_{v}(\widehat{a}),$
$\gamma_{v}(\widehat{b})\in\partial\mathcal{C}_{i}.$ Then, by (\ref{eq:StrongCapConeInvariance})
in Lemma \ref{lem:ImprovedBurnsGerber}, we have $\mathcal{K}_{\gamma'(\widehat{b})}^{u}=d\varphi^{\widehat{b}-a}\mathcal{K}_{\gamma'(a)}^{+}\subset\mathcal{K}_{\gamma'(\widehat{b})}^{+}.$
Thus $\mathcal{K}_{v}^{u}=d\varphi^{-\widehat{b}}\mathcal{K}_{\gamma'(\widehat{b})}^{u}\subset d\varphi^{-\widehat{b}}\mathcal{K}_{\gamma'(\widehat{b})}^{+}\subset\mathcal{K}_{v}^{+},$
where the last inclusion is by (\ref{eq:NegCurvConeInvariance}).
Therefore we have upper semi-continuity at $v_{0}.$
\end{proof}
Definition \ref{def:Z} and Remark \ref{rem:ConeAngleCompactSet}
below will be used to obtain upper bounds on the Lyapunov function
defined in Section \ref{sec:LyapunovLineFields}, while Definitions
\ref{def:RadialVector} and \ref{def:ApproxRadial} and Lemma \ref{lem:H=00003D0}
will be needed in Propositions \ref{pro:Asymptotic} and \ref{pro:BangertGutkinProperty}.
\begin{defn}
\label{def:Z}For $i\in\{1,\dots,q\},$ let $\mathcal{Z}_{i}=(T^{1}(\partial\mathcal{C}_{i}))\cup\{v\in T_{x}^{1}S:x\in\mathcal{D}_{i}\text{\ and\ }v\in\{v_{x,1},v_{x,2}\}\}.$\end{defn}
\begin{rem}
\label{rem:ConeAngleCompactSet}The $\mathcal{K}_{w}^{u}$ cone angle
can approach $0$ as $w$ approaches $\cup_{i=1}^{q}\mathcal{Z}_{i},$
but if $v\in(T^{1}S)\setminus(\cup_{i=1}^{q}\mathcal{Z}_{i}),$ then
there exists an open neighborhood $\mathcal{U}$ of $v$ and an $\alpha=\alpha(\mathcal{U})>0$
such that the cone angle of $\mathcal{K}_{w}^{u}$ is at least $\alpha$
for all $w\in\mathcal{U}.$ Thus, if $\mathcal{W}$ is a compact subset
of $(T^{1}S)\setminus(\cup_{i=1}^{q}\mathcal{Z}_{i}),$ then there
is a positive lower bound for the cone angles of $\mathcal{K}_{w}^{u}$
for $w\in\mathcal{W}.$
\begin{defn}
\label{def:RadialVector}Let $x\in\mathcal{D}=\mathcal{D}_{i}\subset\interior\mathcal{E}=\interior\mathcal{E}_{i},$
for some $i\in\{1,\dots,q\}.$ If $x$ is not the center of $\mathcal{D}$
in the radially symmetric $g$-metric, let $\gamma$ be the $g$-geodesic
from $x$ to a point on $\partial\mathcal{E}$ that is of length dist$_{g}(x,\partial\mathcal{E}).$
Then the unit vector $v_{0}$ in the $h$-metric that is a positive
multiple of $\gamma'(0)$ is called a \emph{radial vector} at $x.$
If $x$ is the center of $\mathcal{D},$ then any $v_{0}\in T_{x}^{1}S$
is called a radial vector at $x.$ (As usual, if we do not specify
the metric, $T^{1}S$ refers to unit vectors for $h.)$
\end{defn}
\end{rem}
\begin{lem}
\label{lem:H=00003D0} There exist positive numbers $\epsilon$ and
$R$ such that if $x\in\cup_{i=1}^{q}\mathcal{D}_{i}$ and $v_{0}$
is a radial vector at $x,$ then for any $v\in T_{x}^{1,h}S$ with
dist$_{h}(v,v_{0})<\epsilon,$ we have $v\notin\cup_{i=1}^{q}\mathcal{Z}_{i}$
and $d\varphi_{v}^{t}(H=0)\subset\interior\mathcal{K}_{\varphi^{t}(v)}^{u},$
for $t\ge R.$ Here $H=0$ means the line in $\mathcal{P}(v)$ with
$H$ coordinate identically $0.$ \end{lem}
\begin{proof}
Let $R_{i}$ be the radius of $\mathcal{E}_{i}$ in the $g$-metric,
and let $R=\max(R_{1},\dots,R_{q}).$ Suppose $x\in\mathcal{D}_{i}$
and let $\mathcal{C}=\mathcal{C}_{i},$ $\mathcal{D}=\mathcal{D}_{i},$
and $\mathcal{E}=\mathcal{E}_{i}.$ Let $\gamma_{g}$ be the (unit
speed) $g$-geodesic with $\gamma_{g}(0)=x$ and $\gamma_{g}'(0)$
a positive multiple of the radial vector $v_{0}.$ Let $\gamma_{h}$
be the $h$-geodesic with $\gamma_{h}(0)=x$ and $\gamma_{h}'(0)=v$
for some $v\in T_{x}^{1}S$ with dist$_{h}(v,v_{0})<\epsilon,$ where
we describe the choice of $\epsilon$ later in the argument. We consider
the solutions $\theta_{g}$ {[}respectively, $\theta_{h}]$ to the
transformed Riccati equation (\ref{eq:TransformedRiccati}) along
$\gamma_{g}$ $[\gamma_{h}]$ with $\theta=\theta_{g}$ $[\theta_{h}]$
and $K=K_{g}$ $[K_{h}],$ the curvature with respect to the $g$
$[h]$ metric. Assume $\theta_{g}$ and $\theta_{h}$ satisfy the
initial condition $\theta_{g}(0)=\pi/2=\theta_{h}(0).$ The condition
$\theta_{h}(0)=\pi/2$ corresponds to the line $H=0$ in $\mathcal{P}(v).$

First we consider the case $x\in\mathcal{C}.$ Then there exist times
$t_{0}=t_{0}(x),$ $t_{1}=t_{1}(x),$ $t_{2}=t_{2}(x),$ $-2R<t_{2}<0\le t_{0}<t_{1}\le R,$
such that for $t\ge0,$ $\gamma_{g}(t)$ exits $\mathcal{C}$ at time
$t_{0}$ and $\gamma_{g}$ exits $\mathcal{E}$ at time $t_{1},$
and for $t\le0,$ $\gamma_{g}(t)$ exits $\mathcal{E}$ at time $t_{2.}$
It follows from Lemma 2.5 in \cite{BurnsGerberII} that $0\le\theta_{g}(t_{0})\le\pi/2.$
Since $K_{g}(t)$ is negative for $t_{0}<t\le t_{1},$ we obtain $0<\theta_{g}(t_{1})<\pi/2$.
Moreover, by a compactness argument, there is a $\delta\in(0,\pi/2)$
such that $\delta<\theta_{g}(t_{1})<(\pi/2)-\delta,$ for $t_{1}=t_{1}(x),$
for all $x\in\mathcal{C}_{i}.$ For $\epsilon$ sufficiently small
and $\mathbb{H}_{2}'$ sufficiently small (i.e., $h$ sufficiently
$C^{2}$ close to $g),$ we obtain $0<\theta_{h}(t_{1})<\pi/2.$ In
addition, we may assume that $\gamma_{h}(t_{1})$ is sufficiently
close to $\gamma_{g}(t_{1})$ that $\gamma_{h}(t_{1})\in S\setminus(\cup_{j=1}^{q}\mathcal{D}_{j}),$
which implies that $\mathcal{K}_{\gamma_{h}'(t_{1})}^{u}=\mathcal{K}_{\gamma_{h}'(t_{1})}^{+}.$
Thus $0<\theta_{h}(t_{1})<\pi/2$ implies that $d\varphi_{v}^{t_{1}}(H=0)\subset\interior\mathcal{K}_{\varphi^{t_{1}}(v)}^{u}.$
By the invariance of the unstable cones it follows that $d\varphi_{v}^{t}(H=0)\subset\interior\mathcal{K}_{\varphi^{t}(v)}^{u},$
for $t\ge R.$ We may also assume that $\theta_{h}(t_{2})\in S\setminus(\cup_{j=1}^{q}\mathcal{D}_{j}),$
which implies that $v\notin\cup_{i=1}^{q}\mathcal{Z}_{i}.$

Now consider the case $x\in\mathcal{D}\setminus\mathcal{C}.$ Again
let $t_{1}=t_{1}(x),$ $0<t_{1}<R,$ be the time at which $\gamma_{g}$
exits $\mathcal{E}.$ Since $K_{g}(t)$ is negative for $0\le t\le t_{1},$
we have $0<\theta_{g}(t_{1})<\pi/2,$ and the rest of the argument
proceeds as in the case $x\in\mathcal{C}.$ \end{proof}
\begin{defn}
\label{def:ApproxRadial} If $x\in\cup_{i=1}^{q}\mathcal{D}_{i}$
, and $v\in T_{x}^{1}(S)$ is such that dist$_{h}(v,v_{0})<\epsilon,$
where $v_{0}$ is a radial vector at $x$ and $\epsilon$ is as in
Lemma \ref{lem:H=00003D0}, then $v$ is said to be an \emph{approximately
radial vector} at $x.$ A $C^{1}$ regular curve $\sigma(t),$ $a_{1}\le t\le a_{2},$
$a_{1}<a_{2},$ in $(T^{1}S)\setminus\big(\cup_{i=1}^{q}T^{1}(\partial\mathcal{C}_{i})\big),$
is defined to be an\emph{ approximately stable {[}approximately unstable{]}
curve} if $\sigma'(t)$ is in $\mathcal{K}_{\sigma(t)}^{s}[\mathcal{K}_{\sigma(t)}^{u}],$
for all $t\in[a_{1},a_{2}].$ \end{defn}
\begin{rem}
\label{rem:ConstantBasepointCurve} It follows from Lemma \ref{lem:H=00003D0}
that if $x\in\cup_{i=1}^{q}\mathcal{D}_{i}$ and $\sigma(s),$ $a_{1}\le s\le a_{2},$
$a_{1}<a_{2},$ is a $C^{1}$ regular constant basepoint curve in
$T_{x}^{1}(S)$ such that each $\sigma(s)$ is an approximately radial
vector, then $\varphi^{t}(\sigma(s)),$ $a_{1}\le s\le a_{2},$ is
an approximately unstable curve for $t\ge R$. To see this, note that
$\sigma'(s)=(\xi_{H}(s),\xi_{V}(s)),$ where $\xi_{H}(s)\equiv0.$
\end{rem}

\section{Lyapunov Function and Line Fields\label{sec:LyapunovLineFields}}

We now introduce a \emph{Lyapunov function} $Q=Q_{w}:\mathcal{P}(w)\to\mathbb{R}$
for each $w\in T^{1}S.$ As we will see in (\ref{eq:LyapunovIncreasing})
below, $Q$ is monotone increasing along the orbits of the geodesic
flow. This Lyapunov function will be used to prove that stable and
unstable cones intersect down to lines in Lemma \ref{lem:LineField}.
Similarly constructed line fields in \cite{BurnsGerberI,BurnsGerberII}
were only shown to exist on some set of full measure, while Lemma
\ref{lem:LineField} shows existence everywhere. Moreover, Lemma \ref{lem:ContinuousLineFields}
shows continuity of these line fields at all vectors in $T^{1}S$
except those that are tangent to the boundary of one of the caps.
Our use of Lyapunov functions and the methods in this section are
based on ideas in \cite{KatokBurns,BarreiraPesin}.

For each $w\in T^{1}S,$ let $U_{i}=U_{i}(w),$ $i=1,2,$ be a basis
for $\mathcal{P}(w)$ such that the unstable cone at $w$ is given
by $\mathcal{K}^{u}(w)=\{\lambda_{1}U_{1}+\lambda_{2}U_{2}:\lambda_{1}\lambda_{2}\ge0\}.$
We require $U_{1},U_{2}$ to be positively oriented (using the orientation
on $\mathcal{P}(w)$ given near the beginning of Section 3), $||U_{1}||=||U_{2}||,$
and the parallelogram determined by $U_{1}$ and $U_{2}$ to have
unit area. This determines $U_{1},U_{2}$ uniquely up to a simultaneous
change of sign in both $U_{1}$ and $U_{2}.$ If $\mathcal{K}_{w}^{u}=\mathcal{K}_{w}^{+},$
we let $U_{1}$ and $U_{2}$ have $H,V$ coordinates $(1,0)$ and
$(0,1),$ respectively. For $w\in T^{1}S$ and $t\in\mathbb{R},$
let

\[
\mathcal{A}=\mathcal{A}(w,t)=\left(\begin{array}{cc}
a(w,t) & b(w,t)\\
c(w,t) & d(w,t)\end{array}\right)\]
be the matrix for $d\varphi^{t}:\mathcal{P}(w)\to\mathcal{P}(\varphi^{t}w)$
with respect to the bases $U_{1}(w),U_{2}(w)$ and $U_{1}(\varphi^{t}(w)),U_{2}(\varphi^{t}(w)).$
Since $d\varphi_{w}^{t}:\mathcal{P}(w)\to\mathcal{P}(\varphi^{t}w)$
is area-preserving and orientation-preserving, $\det\mathcal{A}(w,t)=1.$
The inclusion (\ref{eq:UnstableConeInvariance}) implies that if $t\ge0,$
then either all the entries of $\mathcal{A}(w,t)$ are nonnegative
or all the entries are nonpositive. (Which case occurs may depend
on $t.)$ Moreover if $t>0$ and (\ref{eq:StrictUnstableInvariance})
holds, then all the entries of $\mathcal{A}(w,t)$ are strictly positive
or all the entries are strictly negative. In our calculations, it
does not matter if $\mathcal{A}(w,t)$ is replaced by $-\mathcal{A}(w,t)$.
Therefore, we may assume that for $t\ge0,$ all of the entries of
$\mathcal{A}(w,t)$ are nonnegative. We define a \emph{Lyapunov function}
$Q=Q_{w}:\mathcal{P}(w)\to\mathbb{R}$ by $Q(\lambda_{1}U_{1}(w)+\lambda_{2}U_{2}(w))=\text{sgn}(\lambda_{1}\lambda_{2})\sqrt{|\lambda_{1}\lambda_{2}|}$
. Then $Q_{w}(\xi)\ge0$ if and only if $\xi\in\mathcal{K}_{w}^{u}.$
Thus, by Lemma \ref{lem:DisjointCones}, $Q_{w}(\xi)\le0$ if $\xi\in\mathcal{K}_{w}^{s}.$
We define $F=F_{w}:\mathcal{P}(w)\to\mathbb{R}$ by $F(\lambda_{1}U_{1}(w)+\lambda_{2}U_{2}(w))=\lambda_{1}\lambda_{2}.$

Let $\xi=\lambda_{1}U_{1}(w)+\lambda_{2}U_{2}(w)$ and $a=a(w,t),$
$b=b(w,t),$ $c=c(w,t),$ $d=d(w,t).$ Suppose $t\ge0.$ Assume for
the moment that $\lambda_{1}\lambda_{2}\le0.$ Then\begin{eqnarray}
F(d\varphi_{w}^{t}\xi) & = & (a\lambda_{1}+b\lambda_{2})(c\lambda_{1}+d\lambda_{2})\nonumber \\
 & = & ac\lambda_{1}^{2}+bd\lambda_{2}^{2}+(ad+bc)\lambda_{1}\lambda_{2}\nonumber \\
 & = & ac\lambda_{1}^{2}+bd\lambda_{2}^{2}+2bc\lambda_{1}\lambda_{2}+\lambda_{1}\lambda_{2}\label{eq:FirstEstimate}\\
 & = & ac\lambda_{1}^{2}+bd\lambda_{2}^{2}+2\sqrt{(ad-1)bc}\lambda_{1}\lambda_{2}+\lambda_{1}\lambda_{2}\nonumber \\
 & \ge & ac\lambda_{1}^{2}+bd\lambda_{2}^{2}+2\sqrt{adbc}\lambda_{1}\lambda_{2}+\lambda_{1}\lambda_{2}\nonumber \\
 & = & \big(\sqrt{ac}\lambda_{1}+\sqrt{bd}\lambda_{2}\big)^{2}+\lambda_{1}\lambda_{2}\nonumber \\
 & \ge & \lambda_{1}\lambda_{2}=F(\xi).\nonumber \end{eqnarray}
 But if $\lambda_{1}\lambda_{2}\ge0,$ then (\ref{eq:FirstEstimate})
still holds, and this shows that $F(d\varphi_{w}^{t}\xi)\ge F(\xi).$
Thus we obtain $F(d\varphi_{w}^{t}\xi)\ge F(\xi),$ for all $t\ge0$
and all $\xi\in\mathcal{P}(w).$ This implies that \begin{equation}
Q(d\varphi_{w}^{t}\xi)\ge Q(\xi),\text{\ for\ all\ }t\ge0\text{\ and\ all\ }\xi\in\mathcal{P}(w).\label{eq:LyapunovIncreasing}\end{equation}

If we let \[
\tau(w,t)=\big(2b(t,w)c(t,w)+1\big)^{1/2},\]
then it follows from (\ref{eq:FirstEstimate}) that \begin{equation}
Q(d\varphi_{w}^{t}\xi)\ge\tau(w,t)Q(\xi)\ge0,\text{\ for\ }\xi\in\mathcal{K}_{w}^{u}\text{\ and\ }t\ge0.\label{eq:TauUnstable}\end{equation}
 If $t>0$ and (\ref{eq:StrictUnstableInvariance}) holds, then $\tau(w,t)>1.$

Similarly, \begin{equation}
Q(\xi)\le\tau(w,t)Q(d\varphi_{w}^{t}\xi)\le0,\text{\ for\ }\xi\in d\varphi^{-t}\mathcal{K}_{\varphi^{t}w}^{s}\text{\ and\ }t\ge0.\label{eq:TauStable}\end{equation}

If $\delta>0,$ then there exists $C=C(\delta)>0$ such that if $\mathcal{K}_{w}^{u}=\mathcal{K}_{w}^{+}$
and $\widetilde{\mathcal{K}}_{w}\subset\mathcal{K}_{w}^{+}$ is a
cone such that the slopes of the boundary lines of $\widetilde{\mathcal{K}}_{w}$
are $\delta$ and $1/\delta$ (or if $\mathcal{K}_{w}^{u}=\mathcal{K}_{w}^{+}$
and $\widetilde{\mathcal{K}}_{w}\subset\mathcal{K}_{w}^{-}$ is a
cone such that the slopes of the boundary lines of $\widetilde{\mathcal{K}}_{w}$
are $-\delta$ and $-1/\delta$), then

\begin{equation}
C||\xi||\le|Q(\xi)|,\text{\ for\ all\ }\xi\in\widetilde{\mathcal{K}}_{w}.\label{eq:QLowerBound}\end{equation}
 Also note that for $w$ such that $\mathcal{K}_{w}^{u}=\mathcal{K}_{w}^{+},$
in particular for $w\in T^{1}\mathcal{N}_{0},$ where $\mathcal{N}_{0}=S\setminus(\cup_{i=1}^{q}\mathcal{D}_{i}),$
we have

\begin{equation}
|Q(\xi)|\le||\xi||,\text{\ for\ all\ \ensuremath{\xi\in\mathcal{P}(w).}}\label{eq:QUpperBound}\end{equation}
 If $\mathcal{W}$ is a compact subset of $(T^{1}S)\setminus(\cup_{i=1}^{q}\mathcal{Z}_{i}),$
then it follows from Remark \ref{rem:ConeAngleCompactSet} that there
is a constant $\widetilde{C}=\widetilde{C}(\mathcal{W})>0$ such that
\begin{equation}
|Q(\xi)|\le\widetilde{C}||\xi||,\text{\ for\ all\ }\xi\in\mathcal{P}(w),\text{\ for\ all\ }w\in\mathcal{W}.\label{eq:QUpperBoundOnZComplement}\end{equation}

\begin{lem}
\label{lem:LineField} For all $w\in T^{1}S,$ if we let $E_{w}^{u}$
and $E_{w}^{s}$ be defined by\[
E_{w}^{u}:=\bigcap_{t\ge0}d\varphi^{t}(\mathcal{K}_{\varphi^{-t}w}^{u})\text{{\ and\ }}E_{w}^{s}:=\bigcap_{t\ge0}d\varphi^{-t}(\mathcal{K}_{\varphi^{t}w}^{s}),\]
 then $E_{w}^{u}$ and $E_{w}^{s}$ are lines in $T_{w}(T^{1}S).$
Moreover, if $w\in(T^{1}S)\setminus(\cup_{i=1}^{q}T^{1}(\partial\mathcal{C}_{i})),$
then $E_{w}^{u}\subset\interior\mathcal{K}_{w}^{u}$ and $E_{w}^{s}\subset\interior\mathcal{K}_{w}^{s}.$ \end{lem}
\begin{proof}
Let $w\in T^{1}S.$ We will show that $E_{w}^{u}$ is a line; the
proof that $E_{w}^{s}$ is a line is similar. If $v\in T_{x}^{1}S,$
where $x\in\mathcal{D}$ and $\mathcal{D}=\mathcal{D}_{i}$ for some
$i\in\{1,\dots,q\},$ then either there exists $t>0$ such that $\varphi^{-t}v\in T^{1}\mathcal{N}_{0}$
or $\varphi^{-t}v$ remains in $\mathcal{D}$ for all $t\ge0.$ If
the latter possibility occurs, then either $v$ is tangent to $\partial\mathcal{C},$
where $\mathcal{C}=\mathcal{C}_{i}$ or $\gamma_{v}(t)$ becomes asymptotic
to $\partial\mathcal{C}$ as $t\to-\infty.$ We consider two cases
for the given $w\in T^{1}S.$

\emph{Case 1.} There exist arbitrarily large values of $t$ such that
$\varphi^{-t}w\in T^{1}\mathcal{N}_{0}.$ By the strong convexity
of the disks $\mathcal{D}_{i},$ $i=1,\dots,q,$ (as discussed in
Section 3), there exists $\beta>0$ such that any orbit of the geodesic
flow that enters $T^{1}\mathcal{N}_{0}$ must spend more than $\beta$
units of time within $T^{1}\mathcal{N}_{0}$ before exiting (if it
exits at all). Thus there is a sequence $(t_{n})_{n=1,2,\dots}$ such
that for $n=1,2,\dots,$ we have $-t_{n+1}<-t_{n}-\beta<-t_{n}<0,$
and $\varphi^{t}w\in T^{1}\mathcal{N}_{0}$ for all $t\in[-t_{n}-\beta,-t_{n}].$
It follows from (\ref{eq:StrictUnstableInvariance}) , (\ref{eq:TauUnstable})
and a compactness argument on $\overline{T^{1}\mathcal{N}_{0}}$ that
there exists $\tau_{0}>1$ such that for all $v\in T^{1}S$ with $\varphi^{t}v\in T^{1}\mathcal{N}_{0}$
for all $t\in[-\beta,0]$ we have\begin{equation}
Q_{v}(d\varphi_{\varphi^{-\beta}v}^{\beta}\xi)\ge\tau_{0}Q_{\varphi^{-\beta}v}(\xi),\text{\ for\ all\ }\xi\in\mathcal{K}_{\varphi^{-\beta}v}^{u}.\label{eq:Tau0Factor}\end{equation}
It also follows from (\ref{eq:StrictUnstableInvariance}) and a compactness
argument on $\overline{T^{1}\mathcal{N}_{0}}$ that there exists $\delta>0$
such that for all $v\in T^{1}S$ with $\varphi^{t}v\in T^{1}\mathcal{N}_{0}$
for all $t\in[-\beta,0]$ we have

\begin{equation}
d\varphi^{\beta}\mathcal{K}_{\varphi^{-\beta}v}^{u}\subset\widetilde{\mathcal{K}}_{v},\label{eq:KTildeContainment}\end{equation}
 where $\widetilde{\mathcal{K}}_{v}\subset\mathcal{P}(v)$ is a cone
whose boundary lines have slopes $\delta$ and $1/\delta$ in the
$H,V$ coordinates. Let $C=C(\delta)$ be as in (\ref{eq:QLowerBound}).

Define $C_{\beta}:=\sup\{||d\varphi_{v}^{-\beta}||:v\in T^{1}S\}.$
Let $\xi\in\mathcal{K}_{\varphi^{-t_{n+1}-\beta}w}^{u}$ and suppose
$||\xi||=1.$ Then it follows from (\ref{eq:QLowerBound}) and (\ref{eq:KTildeContainment})
that \[
Q_{\varphi^{-t_{n+1}}w}(d\varphi^{\beta}\xi)\ge C||d\varphi^{\beta}\xi||\ge CC_{\beta}^{-1}.\]
 Moreover, by (\ref{eq:LyapunovIncreasing}), (\ref{eq:QUpperBound}),
and (\ref{eq:Tau0Factor}), we have\begin{equation}
||d\varphi^{t_{n+1}-t_{1}+\beta}\xi||\ge Q_{\varphi^{-t_{1}}w}(d\varphi^{t_{n+1}-t_{1}}d\varphi^{\beta}\xi)\ge\tau_{0}^{n}Q_{\varphi^{-t_{n+1}}w}(d\varphi^{\beta}\xi)\ge\tau_{0}^{n}CC_{\beta}^{-1}.\label{eq:DiskExpansion}\end{equation}
 The intersection of the unit disk in $\mathcal{P}(\varphi^{-t_{n+1}-\beta}w)$
and $\mathcal{K}_{\varphi^{-t_{n+1}-\beta}w}^{u}$ has area $\pi/2$,
and the image of this intersection under $d\varphi^{t_{n+1}-t_{1}+\beta}$
also has area $\pi/2.$ On the other hand, (\ref{eq:DiskExpansion})
implies that this image contains the intersection of the disk of radius
$\tau_{0}^{n}CC_{\beta}^{-1}$ with $d\varphi^{t_{n+1}-t_{1}+\beta}\mathcal{K}_{\varphi^{-t_{n+1}-\beta}w}^{u}.$
Since $\lim_{n\to\infty}\tau_{0}^{n}CC_{\beta}^{-1}=\infty,$ the
cone angle of $d\varphi^{t_{n+1}-t_{1}+\beta}\mathcal{K}_{\varphi^{-t_{n+1}-\beta}w}^{u}$
goes to $0$ as $n\to\infty.$ Thus $E_{\varphi^{-t_{1}}w}^{u}$is
a line, and it follows from the cone invariance (\ref{eq:UnstableConeInvariance})
that $E_{w}^{u}$ is a line.

\emph{Case 2.} Either $w$ is tangent to $\partial\mathcal{C}$ or
$\gamma_{w}(t)$ is asymptotic to $\partial\mathcal{C}$ as $t\to-\infty.$
Then there exists $t_{0}\le0$ such that $\mathcal{K}_{\gamma_{w}(t)}^{u}=\mathcal{K}_{\gamma_{w}(t)}^{+}$
and the curvature $K(\gamma_{w}(t))\le0$ for all $t\le t_{0}.$ As
in case 1, if $E_{\varphi^{t_{0}}w}^{u}$ is a line, then so is $E_{w}^{u}.$
Therefore we may assume that $t_{0}=0.$ Let $\epsilon>0$ and let
$T=T(\epsilon)\le0$ be such that $\ensuremath{-\epsilon\le K(\gamma_{w}(t))\le0}$
for all $t\le T.$ For $B>0,$ let $m_{1,B}(t)$ and $m_{2,B}(t)$
be solutions to the Riccati equation along $\gamma_{w}(t)$ (i.e.,
(\ref{eq:Riccati}) with $w_{0}$ replaced by $w$) such that $m_{1,B}(T-B)=0$
and $\lim_{t\to(T-B)^{+}}m_{2,B}(t)=\infty.$ Since the boundary lines
of $\mathcal{K}_{\varphi^{-B}w}^{u}$ have slopes $0$ and $\infty,$
$m_{1,B}(t)$ and $m_{2,B}(t)$ represent the slopes of the boundary
lines of $d\varphi^{t-(T-B)}(\mathcal{K}_{\varphi^{T-B}w}^{u})$ for
$t>T-B.$ If $K(\gamma_{w_{0}}(t))$ were replaced by the constant
$-\epsilon$ in (\ref{eq:Riccati}), then the solution $\widetilde{m}_{2,B}$
with $\lim_{t\to(T-B)^{+}}\widetilde{m}_{2,B}(t)=\infty$ would be
$\widetilde{m}_{2,B}(t)=\sqrt{\epsilon}\coth\big(\sqrt{\epsilon}(t-(T-B))\big)$
for $t>T-B.$ If $K(\gamma_{w_{0}}(t))$ were replaced by $0,$ then
the solution $\widetilde{m}_{1,B}$ with $\widetilde{m}_{1,B}(T-B)=0$
would be identically $0$ for all $t.$ Therefore, by a comparison
lemma (see, e.g., \cite{BallmannBrinBurns}), we obtain $\lim\sup_{B\to\infty}m_{2,B}(T)\le\lim_{B\to\infty}\widetilde{m}_{2,B}(T)=\sqrt{\epsilon}$
and $\lim\inf_{B\to\infty}m_{1,B}(T)\ge\lim_{B\to\infty}\widetilde{m}_{1,B}(T)=0.$
Thus $E_{\varphi^{T}w}^{u}$ is contained in a cone within $\mathcal{K}_{\varphi^{T}w}^{+}$
that is bounded by lines of slopes $0$ and $\sqrt{\epsilon}.$ By
subtracting the equation (\ref{eq:Riccati}) with $m=m_{1,B}$ from
the equation (\ref{eq:Riccati}) with $m=m_{2,B},$ we obtain \begin{equation}
(m_{2,B}-m_{1,B})'(t)=-[(m_{1,B}+m_{2,B})(m_{2,B}-m_{1,B})](t)<0,\text{\ for\ }t\in(-B,0].\label{eq:RiccatiDifference}\end{equation}
 Thus $\lim\sup_{B\to\infty}(m_{2,B}-m_{1,B})(0)\le\sqrt{\epsilon},$
which implies that $E_{w}^{u}$ is contained in a cone within $\mathcal{K}_{w}^{+}$
that is bounded by lines whose slopes differ by at most $\sqrt{\epsilon}.$
Since $\epsilon$ was arbitrary, $E_{w}^{u}$ is a line.

In both cases, if $w$ is not tangent to the boundary of a cap, then
there exist $r_{1}>r_{2}>0$ such that $d\varphi^{r_{1}-r_{2}}\mathcal{K}_{\varphi^{-r_{1}}w}^{u}\subset\interior\mathcal{K}_{\varphi^{-r_{2}}w}^{u}.$
Thus $E_{w}^{u}\subset\varphi^{r_{2}}(\interior\mathcal{K}_{\varphi^{-r_{2}}w}^{u})\subset\interior\mathcal{K}_{w}^{u}.$
A similar argument shows that if $w$ is not tangent to the boundary
of a cap, then $E_{w}^{s}\subset\interior\mathcal{K}_{w}^{u}.$ \end{proof}
\begin{defn}
The line fields $v\to E_{v}^{s}$ and $v\to E_{v}^{u}$ on $T^{1}S$
obtained in Lemma \ref{lem:LineField} will be called the \emph{stable}
\emph{and} \emph{unstable line fields}, respectively. Note that $d\varphi^{t}E_{v}^{u}=E_{\varphi^{t}v}^{u}$
and $d\varphi^{t}E_{v}^{s}=E_{\varphi^{t}v}^{s}$ for $v\in T^{1}S$
and $t\in\mathbb{R}.$ A $C^{1}$ regular curve $\sigma(t)$, $a_{1}\le t\le a_{2},$
$a_{1}<a_{2},$ in $(T^{1}S)\setminus(\cup_{i=1}^{q}T^{1}(\partial\mathcal{C}_{i}))$
is defined to be a \emph{stable {[}unstable{]} curve} if $\sigma'(t)\in E_{\sigma(t)}^{s}[E_{\sigma(t)}^{u}]$
for all $t\in[a_{1},a_{2}].$
\end{defn}
The following lemma allows us to integrate the stable and unstable
line fields to obtain stable and unstable curves.
\begin{lem}
\label{lem:ContinuousLineFields} The stable and unstable line fields
given in Lemma \ref{lem:LineField} are continuous on $(T^{1}S)\setminus(\cup_{i=1}^{q}T^{1}(\partial\mathcal{C}_{i}))$
\textup{. }\end{lem}
\begin{proof}
We will give the proof for the unstable line field. The proof for
the stable line field is similar.

In order to compare a cone in $\mathcal{P}(w_{1})$ with a cone in
$\mathcal{P}(w_{2}),$ or a line through the origin in $\mathcal{P}(w_{1})$
with a line through the origin in $\mathcal{P}(w_{2})$, we will use
$H,V$ coordinates on both $\mathcal{P}(w_{1})$ and $\mathcal{P}(w_{2})$
to identify $\mathcal{P}(w_{1})$ with $\mathcal{P}(w_{2}).$

Suppose $v\in T^{1}S$ and $v$ is not tangent to the boundary of
a cap. Let $\epsilon>0$ and let $\mathcal{K}_{v,\epsilon}$ be the
closed cone in $\mathcal{P}(v)$ centered at $E_{v}^{u}$ and of cone
angle $\epsilon$ in the $H,V$ coordinate system. By Lemma \ref{lem:LineField}
and (\ref{eq:UnstableConeInvariance}), there exists $T>0$ such that
$d\varphi^{T}\mathcal{K}_{\varphi^{-T}v}^{u}\subset\interior\mathcal{K}_{v,\epsilon}.$
By the continuity of $d\varphi^{T},$ there exists a closed cone $\tilde{\mathcal{K}}_{\varphi^{-T}v}\subset\mathcal{P}(\varphi^{-T}v)$
such that $\mathcal{K}_{\varphi^{-T}v}^{u}\subset\interior\tilde{\mathcal{K}}_{\varphi^{-T}v}$
and $d\varphi^{T}\tilde{\mathcal{K}}_{\varphi^{-T}v}\subset\interior\mathcal{K}_{v,\epsilon}.$
From the continuity of $\varphi^{-T}$ at $v$ and the upper semi-continuity
of the unstable cone field at $\varphi^{-T}v,$ we know that for $w\in T^{1}S$
sufficiently close to $v,$ we have $\mathcal{K}_{\varphi^{-T}w}^{u}\subset\tilde{\mathcal{K}}_{\varphi^{-T}w},$
where $\tilde{\mathcal{K}}_{\varphi^{-T}w}\subset\mathcal{P}(\varphi^{-T}w)$
is a copy of $\tilde{\mathcal{K}}_{\varphi^{-T}v}$ (using the $H,V$
coordinates as described in the preceding paragraph). Moreover, by
the continuity of $\varphi^{-T}$ and $d\varphi^{T},$ for $w$ sufficiently
close to $v,$ we have $d\varphi^{T}(\tilde{\mathcal{K}}_{\varphi^{-T}w})\subset\interior\mathcal{K}_{w,\epsilon},$
where $\mathcal{K}_{w,\epsilon}\subset\mathcal{P}(w)$ is a copy of
$\mathcal{K}_{v,\epsilon}.$ Therefore \[
E_{w}^{u}\subset d\varphi^{T}(\mathcal{K}_{\varphi^{-T}w}^{u})\subset d\varphi^{T}(\tilde{\mathcal{K}}_{\varphi^{-T}w})\subset\interior\mathcal{K}_{w,\epsilon},\]
 which implies that $E_{w}^{u}$ makes angle less than $\epsilon$
with $E_{v}^{u}.$
\end{proof}

\section{Tubular Neighborhoods of tr$(\rho)$ And Their Lifts\label{sec:TubularNeighborhoods}}

For the rest of this paper, we let $\mathcal{N}=S\setminus(\cup_{i=1}^{q}\mathcal{C}_{i}),$
$\mathcal{N}_{0}=S\setminus(\cup_{i=1}^{q}\mathcal{D}_{i}),$ and
we let $\mathcal{N}_{1}$ and $\mathcal{N}_{2}$ be open subsets of
$S$ such that $\overline{\mathcal{N}_{2}}\subset\mathcal{N}_{1},$
$\overline{\mathcal{N}_{1}}\subset\mathcal{N}_{0},$ and the closed
geodesic $\rho:[0,L]\to\mathcal{N}$ described in Section \ref{sec:Construction-of-Metrics}
has tr$(\rho)\subset\mathcal{N}_{2}.$ We let $(\tau_{1},\tau_{2})$
be Fermi coordinates along $\rho,$ where $\tau_{1}\in\mathbb{R}/L\mathbb{Z}$
is the coordinate along $\rho$ and $\tau_{2}\in[-\epsilon_{0},\epsilon_{0}]$
is the coordinate along geodesics perpendicular to $\rho.$ Here $\epsilon_{0}>0$
is chosen sufficiently small so that all points with Fermi coordinates
in $(\mathbb{R}/L\mathbb{Z})\times[-\epsilon_{0},\epsilon_{0}]$ are
contained in $\mathcal{N}_{2}.$ For $0<\epsilon\le\epsilon_{0},$
let\begin{equation}
F(\epsilon)=\{p\in S:\text{dist}(p,\text{tr}(\rho)\le\epsilon\}.\label{eq:F}\end{equation}
 Each point in $F(\epsilon)$ has Fermi coordinates $(\tau_{1},\tau_{2})$
in $(\mathbb{R}/L\mathbb{Z})\times[-\epsilon_{0},\epsilon_{0}]$,
but if $\rho$ is not simple, then some of the points in $F(\epsilon)$
will have more than one such representation in Fermi coordinates.
In order to handle the case in which $\rho$ is not simple, we let
\begin{equation}
\widehat{F}(\epsilon)=(\mathbb{R}/L\mathbb{Z})\times[-\epsilon,\epsilon],\label{eq:FHat}\end{equation}
 and let $\widehat{\pi}:\widehat{F}(\epsilon)\to F(\epsilon)$ be
the projection that takes $(\tau_{1},\tau_{2})$ to the point in $S$
with Fermi coordinates $(\tau_{1},\tau_{2}).$ Define $\widehat{h}=\widehat{\pi}^{*}h$
to be the covering metric. Let $\widehat{\rho}:[0,L]\to\widehat{F}(\epsilon)$
be the \emph{simple} closed geodesic in $(\widehat{F}(\epsilon),\widehat{h})$
that is the lift of $\rho.$ We also define \begin{equation}
\widetilde{F}(\epsilon)=\mathbb{R}\times[-\epsilon,\epsilon],\label{eq:FTilde}\end{equation}
 which is the universal covering space of $\widehat{F}(\epsilon).$
Let $\widetilde{\pi}:\widetilde{F}(\epsilon)\to\widehat{F}(\epsilon)$
be the covering map, and let $\widetilde{h}=\widetilde{\pi}^{*}\widehat{h}$
be the covering metric. We let $\widetilde{\rho}:\mathbb{R}\to\widetilde{F}(\epsilon)$
be the geodesic in $(\widetilde{F}(\epsilon),\widetilde{h})$ that
is the lift of $\widehat{\rho}.$

If $(\tau_{1},\tau_{2})$ are Fermi coordinates along $\widehat{\rho}$
or $\widetilde{\rho}$ and $\gamma:I\to\widehat{F}(\epsilon)$ or
$\gamma:I\to\widetilde{F}(\epsilon)$ is a geodesic whose trace is
contained in the region in which $\tau_{2}\ne0,$ then the negative
curvature of $(\widehat{F}(\epsilon),\widehat{h})$ and $(\widetilde{F}(\epsilon),\widetilde{h})$
implies that $t\mapsto|\tau_{2}(\gamma(t))|$ is a strictly convex
function.

It follows from the simple connectivity of $\widetilde{F}(\epsilon)$
and the negative curvature of $(\widetilde{F}(\epsilon),\widetilde{h})$
that for each $\widetilde{p}\in\widetilde{F}(\epsilon),$ there is
a unique vector $\widetilde{Z}(\widetilde{p})\in T_{\tilde{p}}^{1,\tilde{h}}(\widetilde{F}(\epsilon))$
such that $\gamma_{\tilde{Z}(\tilde{p})}$ is asymptotic to $\widetilde{\rho}$
as $t\to\infty.$ We let $\widehat{Z}$ be the (unique) vector field
on $\widehat{F}(\epsilon)$ obtained by applying $d\widetilde{\pi}$
to $\widetilde{Z}.$ Then for each $\widehat{p}\in\widehat{F}(\epsilon),$
$\gamma_{\widehat{Z}(\widehat{p})}$ is asymptotic to $\widehat{\rho}$
as $t\to\infty.$

We will call $\widehat{Z}$ and $\widetilde{Z}$ \emph{asymptotic}
\emph{vector fields} for $\widehat{\rho}$ and $\widetilde{\rho},$
respectively. These vector fields can be obtained by applying the
geodesic flow to a stable horocycle through a tangent vector to $\widehat{\rho}$
or $\widetilde{\rho}.$ Such horocycles are known to be $C^{\infty}$
\cite{Anosov}, because $(\widehat{F}(\epsilon),\widehat{h})$ can
be extended to a closed surface of negative curvature. It follows
that $\widehat{Z}$ and $\widetilde{Z}$ are $C^{\infty}$ vector
fields. In this paper, we will only use the fact that they are continuous.

\section{Proof of Total Insecurity of $(S,h)$}

Let $\mathcal{P}(w),$ $w\in T^{1}S,$ be the two-dimensional distribution
on $T^{1}S$ defined in (\ref{eq:Distribution}), and let $Q=Q_{w}:\mathcal{P}(w)\to\mathbb{R}$
be the Lyapunov function defined at the beginning of Section \ref{sec:LyapunovLineFields}.
Let $\mathcal{N},$ $\mathcal{N}_{0},$ $\mathcal{N}_{1},$ and $\mathcal{N}_{2}$
be as in Section \ref{sec:TubularNeighborhoods}.
\begin{defn}
\label{def:LyapunovLength}If a $C^{1}$ regular curve $\sigma(s)$,
$a_{1}\le s\le a_{2},$ $a_{1}<a_{2},$ in $T^{1}S$ is everywhere
tangent to the distribution $\mathcal{P}$, then we define the \emph{Lyapunov
length} of $\sigma$ as $\mathcal{L}_{Q}(\sigma)=\int_{a_{1}}^{a_{2}}|Q(\sigma'(s)|\ ds$.
We let $\mathcal{L}(\sigma)=\int_{a_{1}}^{a_{2}}||\sigma'(s)||\ ds$
be the usual length of $\sigma$. \end{defn}
\begin{rem}
\label{rem:LyapunovLengthBound}The Lyapunov length, like the usual
length, is independent of the parametrization of $\sigma,$ because
$Q$ is homogeneous of degree 1. If $\sigma$ is everywhere tangent
to $\mathcal{P}$, and for each tangent vector $w=\sigma'(s)$ of
$\sigma$ we have $\mathcal{K}_{w}^{u}=\mathcal{K}_{w}^{+}$ (in particular
if tr$(\sigma)\subset\mathcal{N}_{0}),$ then by (\ref{eq:QUpperBound}),
we know that $\mathcal{L}_{Q}(\sigma)\le\mathcal{L}(\sigma).$ If
$\sigma$ is as in Definition \ref{def:LyapunovLength} and $\sigma(s)\in(T^{1}S)\setminus(\cup_{i=1}^{q}\mathcal{Z}_{i})$
for all $s\in[a_{1},a_{2}],$ then it follows from (\ref{eq:QUpperBoundOnZComplement})
that $\mathcal{L}_{Q}(\sigma)<\infty.$
\end{rem}
Lemma \ref{lem:StableCurveContraction} shows that the length of certain
stable or approximately stable curves goes to zero under the application
of $d\varphi^{t}$ as $t$ goes to infinity along certain sequences.
The uniformity of the contraction in part (1) of Lemma \ref{lem:StableCurveContraction}
allows us to obtain expansion of approximately unstable curves in
Corollary \ref{cor:UnstableCurveExpansion}.
\begin{lem}
\label{lem:StableCurveContraction} Suppose $\sigma(s)$, $s\in[0,a]$,
$a>0,$ is a $C^{1}$ regular curve in $(T^{1}S)\setminus$ $(\cup_{i=1}^{q}T^{1}(\partial\mathcal{C}_{i}))$
that is everywhere tangent to the distribution $\mathcal{P}.$ Let
$v_{0}=\sigma(s_{0}),$ for some $s_{0}\in[0,a]$. The two statements
below are true for $\sigma.$
\begin{enumerate}
\item For every $\epsilon>0,$ there exists $M=M(\epsilon)>1,$ independent
of the choice of $\sigma,$ such that the following holds: If $m\ge M$
is an integer and there exists a finite sequence $(t_{n})_{n=1,2,\dots,m+1}$
such that $t_{1}\ge0,$ $t_{n+1}\ge1+t_{n}$ for $n=1,2,\dots,m,$
tr$(\varphi^{t_{n}}(\sigma))\subset T^{1}\mathcal{N}_{1},$ for $n=1,2,\dots,m,$
and $(\varphi^{t_{m+1}}(\sigma))$ is an approximately stable curve,
then $\mathcal{L}(\varphi^{t_{m}}(\sigma))<\epsilon\mathcal{L}(\varphi^{t_{1}}(\sigma))$.
\item If $\sigma$ is a stable curve with $\mathcal{L}_{Q}(\sigma)<\infty,$
and there exists a sequence $(t_{n})_{n=1,2,\dots}$ such that $t_{1}\ge0$,
$t_{n+1}\ge1+t_{n}$ for $n=1,2,\dots,$ and $\varphi^{t_{n}}(v_{0})\in T^{1}\mathcal{N}_{2}$
for $n=1,2,\dots,$ then $\lim_{n\to\infty}\mathcal{L}(\varphi^{t_{n}}(\sigma))=0.$
\end{enumerate}
\end{lem}
\begin{proof}
Since $\overline{\mathcal{N}_{1}}\subset\mathcal{N}_{0},$ there exists
$\eta\in(0,1)$ such that for every $v\in\overline{T^{1}\mathcal{N}_{1}},$
$\gamma_{v}([0,\eta])$ lies in $\mathcal{N}_{0}.$ By (\ref{eq:StrictStableInvariance})
and a compactness argument, there exists $\delta\in(0,1)$ such that
for $v\in\overline{T^{1}\mathcal{N}_{1}},$ $d\varphi^{-\eta}(\mathcal{K}_{\varphi^{\eta}v}^{s})$
$\subset\widetilde{\mathcal{K}}_{v}\subset\mathcal{K}_{v}^{-},$ where
$\widetilde{\mathcal{K}}_{v}$ is bounded by lines of slopes $-\delta$
and $-1/\delta$ in the $H,V$ coordinate system on $\mathcal{P}(v).$
If $\beta\ge1>\eta,$ then $d\varphi^{-\beta}(\mathcal{K}_{\varphi^{\beta}v}^{s})\subset d\varphi^{-\eta}(\mathcal{K}_{\varphi^{\eta}v}^{s}).$
Thus, by (\ref{eq:QLowerBound}) and (\ref{eq:QUpperBound}) there
exists a constant $C>0$ such that if $\xi\in d\varphi^{-\beta}(\mathcal{K}_{\varphi^{\beta}v}^{s})$
for some $\beta\ge1$ and some $v\in\overline{T^{1}\mathcal{N}_{1}},$
then \begin{equation}
C||\xi||\le|Q(\xi)|\le||\xi||.\label{eq:XiLengthEstimate}\end{equation}

By (\ref{eq:StrictUnstableInvariance}) and (\ref{eq:TauStable}),
there exists $\kappa\in(0,1)$ such that if $\xi\in d\varphi^{-\eta}(\mathcal{K}_{\varphi^{\eta}v}^{s})$
for some $v\in\overline{T^{1}\mathcal{N}_{1}},$ then\[
|Q(d\varphi^{\eta}\xi)|\le\kappa|Q(\xi)|.\]
 If $\beta\ge1$ and $\xi\in d\varphi^{-\beta}(\mathcal{K}_{\varphi^{\beta}v}^{s})$
for some $v\in\overline{T^{1}\mathcal{N}_{1}},$ then $|Q(d\varphi^{\beta}\xi)|\le|Q(d\varphi^{\eta}\xi)|$
and we obtain \[
|Q(d\varphi^{\beta}\xi)|\le\kappa|Q(\xi)|.\]
Thus, if $\widetilde{\sigma}$ is a stable curve in $\overline{T^{1}\mathcal{N}_{1}},$
or more generally, if $\widetilde{\sigma}$ is in $\overline{T^{1}\mathcal{N}_{1}}$
and $\varphi^{\beta}\widetilde{\sigma}$ is an approximately stable
curve for some $\beta\ge1,$ then \begin{equation}
C\mathcal{L}(\widetilde{\sigma})\le\mathcal{L}_{Q}(\widetilde{\sigma})\le\mathcal{L}(\widetilde{\sigma})\label{eq:QLengthComparison}\end{equation}
 and \begin{equation}
\mathcal{L}_{Q}(\varphi^{\beta}(\widetilde{\sigma}))\le\kappa\mathcal{L}_{Q}(\widetilde{\sigma}).\label{eq:QContraction}\end{equation}

Let $\epsilon>0$ and take $M>1$ sufficiently large so that $\kappa^{M-1}<C\epsilon.$
Suppose $m\ge M$ is an integer and $(t_{n})_{n=1,\dots,m+1}$ and
$\sigma$ are as in the hypothesis of part (1). Since $\varphi^{t_{m+1}}(\sigma)$
is approximately stable and tr$(\varphi^{t_{m}}(\sigma))\subset T^{1}\mathcal{N}_{1},$
it follows from (\ref{eq:QLengthComparison}) that $C\mathcal{L}(\varphi^{t_{m}}(\sigma))\le\mathcal{L}_{Q}(\varphi^{t_{m}}(\sigma)).$
Note that (\ref{eq:StableConeInvariance}) implies that $\varphi^{t_{m}-t}(\sigma)$
is approximately stable for $t\ge0.$ Applying (\ref{eq:QContraction})
to $\widetilde{\sigma}=\varphi^{t_{1}}(\sigma),\dots,\varphi^{t_{m-1}}(\sigma),$
and $\beta=t_{2}-t_{1},\dots,t_{m}-t_{m-1},$ respectively, we obtain
$\mathcal{L}_{Q}(\varphi^{t_{m}}(\sigma))\le\kappa^{m-1}\mathcal{L}_{Q}(\varphi^{t_{1}}(\sigma)).$
Since $\varphi^{t_{2}}(\sigma)$ is approximately stable, (\ref{eq:QLengthComparison})
implies that $\mathcal{L}_{Q}(\varphi^{t_{1}}(\sigma))\le\mathcal{L}(\varphi^{t_{1}}(\sigma)).$
To summarize, we have \[
C\mathcal{L}(\varphi^{t_{m}}(\sigma))\le\mathcal{L}_{Q}(\varphi^{t_{m}}(\sigma))\le\kappa^{m-1}\mathcal{L}_{Q}(\varphi^{t_{1}}(\sigma))\le\kappa^{m-1}\mathcal{L}(\varphi^{t_{1}}(\sigma))<C\epsilon\mathcal{L}(\varphi^{t_{1}}(\sigma)),\]
 and this completes the proof of part (1).

It follows from (\ref{eq:QLengthComparison}) that there exists $\widetilde{\epsilon}>0$
such that any stable curve $\widetilde{\sigma}$ that contains a vector
in $T^{1}\mathcal{N}_{2}$ and has $\mathcal{L}_{Q}(\widetilde{\sigma})\le\widetilde{\epsilon}$
must be contained in $T^{1}\mathcal{N}_{1}.$ Suppose now that $\sigma$
and the sequence $(t_{n})_{n=1,2,\dots}$ are as in the hypothesis
of part (2). If $\mathcal{L}_{Q}(\varphi^{t_{n}}(\sigma))\le\widetilde{\epsilon}$
for some $n\ge1,$ then $\varphi^{t_{n}}(\sigma)\subset T^{1}\mathcal{N}_{1}$
and $\mathcal{L}(\varphi^{t_{n}}(\sigma))\le(1/C)\mathcal{L}_{Q}(\varphi^{t_{n}}(\sigma)).$
Therefore to prove part (2), it suffices to show that $\lim_{n\to\infty}$
$\mathcal{L}_{Q}(\varphi^{t_{n}}(\sigma))=0.$

We may assume that $v_{0}$ is an endpoint of $\sigma,$ say $v_{0}=\sigma(0).$
Let $0=s_{0}<s_{1}<\cdots<s_{k}=a$ be chosen so that $\mathcal{L}_{Q}(\varphi^{t_{1}}(\sigma|[s_{\ell},s_{\ell+1}]))<\widetilde{\epsilon}/2$
for $\ell=0,\dots,k-1.$ Let $\sigma_{\ell}=\sigma|[0,s_{\ell}]$
for $\ell=1,\dots,k.$ We will use induction on $\ell$ to show that
$\lim_{n\to\infty}\mathcal{L}_{Q}(\varphi^{t_{n}}(\sigma_{\ell}))$
$=0$ for $l=1,\dots,k.$ Since $\sigma_{k}=\sigma,$ the conclusion
of part (2) will follow.

Since $\mathcal{L}_{Q}(\varphi^{t}(\tilde{\sigma}))$ is a nonincreasing
function of $t$ for any stable curve $\tilde{\sigma},$ we have $\mathcal{L}_{Q}(\varphi^{t_{n}}(\sigma_{1}))\le\mathcal{L}_{Q}(\varphi^{t_{1}}(\sigma_{1}))<\widetilde{\epsilon}/2$
for all $n\ge1.$ Thus tr$(\varphi^{t_{n}}(\sigma_{1}))\subset T^{1}\mathcal{N}_{1}$
for all $n\ge1.$ Therefore part (1) implies that $\lim_{n\to\infty}\mathcal{L}(\varphi^{t_{n}}(\sigma_{1}))=0.$
Since $\mathcal{L}_{Q}(\varphi^{t_{n}}(\sigma_{1}))\le\mathcal{L}(\varphi^{t_{n}}(\sigma_{1})),$
we have $\lim_{n\to\infty}\mathcal{L}_{Q}(\varphi^{t_{n}}(\sigma_{1}))=0.$

Now suppose that $\ell\in\{1,\dots,k-1\}$ and assume that $\lim_{n\to\infty}\mathcal{L}_{Q}(\varphi^{t_{n}}(\sigma_{\ell}))=0.$
Let $M_{\ell}$ be sufficiently large so that $\mathcal{L}_{Q}(\varphi^{t_{n}}(\sigma_{\ell}))<\widetilde{\epsilon}/2$
for $n\ge M_{\ell}.$ Since $\mathcal{L}_{Q}(\varphi^{t_{n}}(\sigma|[s_{\ell},s_{\ell+1}]))\le\mathcal{L}_{Q}(\varphi^{t_{1}}(\sigma|[s_{\ell},s_{\ell+1}]))<\widetilde{\epsilon}/2$
for all $n\ge1,$ it follows that $\mathcal{L}_{Q}(\varphi^{t_{n}}(\sigma_{\ell+1}))$
$<\widetilde{\epsilon}$ for $n\ge M_{\ell}.$ Then tr$(\varphi^{t_{n}}(\sigma_{\ell+1}))\subset T^{1}\mathcal{N}_{1}$
for $n\ge M_{\ell}.$ By applying part (1), we obtain $\lim_{n\to\infty}\mathcal{L}_{Q}(\varphi^{t_{n}}(\sigma_{\ell+1}))=0,$
which completes the inductive proof.$ $\end{proof}
\begin{cor}
\label{cor:UnstableCurveExpansion}There exists $r_{0}>0$ such that
if $\sigma(s),$ $s\in[0,a],$ $a>0,$ is an approximately unstable
curve and there exist $v_{0}=\sigma(s_{0})$ for some $s_{0}\in[0,a]$
and a sequence $(t_{n})_{n=1,2,\dots}$ with $t_{1}\ge1,$ $t_{n+1}\ge1+t_{n}$
for $n=1,2,\dots,$ and $\varphi^{t_{n}}(v_{0})\in T^{1}\mathcal{N}_{2}$
for $n=1,2,\dots,$ then $\liminf_{n\to\infty}\mathcal{L}(\varphi^{t_{n}}(\sigma))>r_{0}.$\end{cor}
\begin{proof}
Let $b=\text{dist}(T^{1}\mathcal{N}_{2},\partial(T^{1}\mathcal{N}_{1}))>0$
and let $r_{0}=\min(1/2,b/2).$ Let $\sigma(s),$ $s\in[0,a],$ $a>0,$
be an approximately unstable curve and let $v_{0}$ and $(t_{n})_{n=1,2,\dots}$
be as in the hypothesis of the corollary. If tr$(\varphi^{t_{n}}(\sigma))$
is not contained in $T^{1}\mathcal{N}_{1},$ then $\mathcal{L}(\varphi^{t_{n}}(\sigma))\ge b>r_{0}.$
Thus, we may assume that tr$(\varphi^{t_{n}}(\sigma))\subset T^{1}\mathcal{N}_{1}$
for $n=1,2,\dots.$

If $p(s)$ is a curve in $S$ and $\alpha(s)$ is a curve in $T^{1}S$
such that $\alpha(s)\in T_{p(s)}^{1}S,$ then we let $\alpha_{\text{rev }}$
denote the curve in $T^{1}S$ with the same basepoints as $\alpha,$
but the vectors $\alpha(s)\in T_{p(s)}^{1}S$ are replaced by $-\alpha(s)\in T_{p(s)}^{1}S.$
Since $\mathcal{K}_{v}^{u}=\mathcal{K}_{v}^{+}$ and $\mathcal{K}_{v}^{s}=\mathcal{K}_{v}^{-}$
for $v\in T^{1}\mathcal{N}_{1},$ a $C^{1}$ curve $\alpha$ in $T^{1}\mathcal{N}_{1}$
is approximately stable {[}unstable{]} if and only if $\alpha_{\text{rev }}$
is approximately unstable {[}stable{]}. Note that $\varphi^{t}((\varphi^{t}(\alpha))_{\text{rev }})=\alpha_{\text{rev}}$
for all $t\in\mathbb{R}.$

Since $\sigma$ is approximately unstable, so is $\varphi^{t_{n}}(\sigma),$
for $n=1,2,\dots.$ Let $\epsilon=\mathcal{L}(\varphi^{t_{1}}(\sigma))>0,$
and let $M=M(\epsilon)$ be as in part (1) of Lemma \ref{lem:StableCurveContraction}.
Suppose there exists $m\ge M$ such that $\mathcal{L}(\varphi^{t_{m}}(\sigma))<2r_{0}\le1.$
Let $\widetilde{\sigma}=(\varphi^{t_{m}}(\sigma))_{\text{rev}}$ and
let $t_{0}=0.$ Then for $0\le k\le m,$ we have $(\varphi^{t_{k}}(\sigma))_{\text{rev}}=\varphi^{t_{m}-t_{k}}(\widetilde{\sigma}).$
Applying part (1) of Lemma \ref{lem:StableCurveContraction} to $\widetilde{\sigma}$
(instead of $\sigma),$ using the sequence $0,$ $t_{m}-t_{m-1},$
$t_{m}-t_{m-2},\dots,t_{m}-t_{1},t_{m}$ (instead of $t_{1},t_{2},\dots,t_{m+1}),$
and noting that $\varphi^{t_{m}}(\widetilde{\sigma})=\sigma_{\text{rev}}$
is approximately stable, we obtain $\mathcal{L}((\varphi^{t_{1}}(\sigma))_{\text{rev}})=\mathcal{L}(\varphi^{t_{m}-t_{1}}(\widetilde{\sigma}))<\epsilon\mathcal{L}(\widetilde{\sigma}))<\epsilon,$
which is a contradiction, since $\mathcal{L}((\varphi^{t_{1}}(\sigma))_{\text{rev}})=\mathcal{L}(\varphi^{t_{1}}(\sigma))=\epsilon.$
Thus $\mathcal{L}(\varphi^{t_{m}}(\sigma))$ $\ge2r_{0}$ for all
$m\ge M.$
\end{proof}
Lemma \ref{lem:FermiCurvature} below provides a comparison between
the Euclidean curvature and the curvature within a manifold $(M,\overline{h})$
for a given curve $\alpha$ in a neighborhood of a geodesic $\gamma.$
In our application of this lemma, $\gamma$ will be replaced by $\widehat{\rho},$
and $\overline{h}$ will be replaced by $\widehat{h}.$ If $\alpha(t)$
is a regular $C^{2}$ curve in the Euclidean plane and $N(t)$ is
a unit normal field along $\alpha(t)$, then the \emph{signed curvature
of $\alpha$ with respect to $N$} is defined by \begin{equation}
\kappa(t)=\frac{-\alpha''(t)\cdot N(t)}{|\alpha'(t)|^{2}}\label{eq:EuclidCurvature}\end{equation}
The choice of sign is consistent with that in equation (\ref{eq:DefSignedCurvature}).

Suppose $a<b$ and $\gamma:[a,b]\to M$ is a geodesic in a complete
Riemannian surface $(M,\overline{h})$ such that $\gamma$ is either
one-to-one on $[a,b]$ or $\gamma$ is a closed geodesic that is one-to-one
on $[a,b).$ In the former case, let $I=[a,b],$ and in the latter
case, let $I=[a,b]/\sim,$ where $a\sim b.$ Let $(\tau_{1},\tau_{2})$
be Fermi coordinates along $\gamma,$ where $\tau_{1}$ is the parameter
along $\gamma$ and $\tau_{2}$ is the parameter along geodesics perpendicular
to $\gamma.$ Let $\epsilon_{1}>0$ be sufficiently small so that
points that are within distance at most $\epsilon_{1}$ from tr$(\gamma)$
have a unique representation in Fermi coordinates $(\tau_{1},\tau_{2})\in I\times[-\epsilon_{1},\epsilon_{1}].$
For $\epsilon$ with $0<\epsilon\le\epsilon_{1},$ define\begin{equation}
\overline{F}(\epsilon)=\{(\tau_{1},\tau_{2}):\tau_{1}\in I,-\epsilon\le\tau_{2}\le\epsilon\}.\label{eq:FOverline}\end{equation}

On $\overline{F}(\epsilon)$ we have the Riemannian metric induced
by $\overline{h}$ through the identification of points within distance
at most $\epsilon$ of tr$(\gamma)$ in $(M,\overline{h})$ and their
Fermi coordinates $(\tau_{1},\tau_{2}).$ We also have the Euclidean
metric $\langle\partial/\partial\tau_{i},\partial/\partial\tau_{j}\rangle=\delta_{ij},$
for $1\le i,j\le2$ on $\overline{F}(\epsilon).$ For a regular $C^{2}$
curve $\alpha(s),$ $-\ell\le s\le\ell,$ in $\overline{F}(\epsilon)$
we may compare the signed curvature of $\alpha$ in these two metrics.
We will use the subscripts $\overline{h}$ and $e$ to distinguish
inner products with respect to the metric $\overline{h}$ and the
Euclidean metric.
\begin{lem}
\label{lem:FermiCurvature}Suppose $\gamma:[a,b]\to M$ is a geodesic
on a complete Riemannian surface $(M,\overline{h})$ such that $\gamma$
is either one-to-one on $[a,b]$ or $\gamma$ is a closed geodesic
that is one-to-one on $[a,b).$ Suppose $I$, $\epsilon_{1},$ Fermi
coordinates $(\tau_{1},\tau_{2})$ along $\gamma,$ and $\overline{F}(\epsilon)$
for $0<\epsilon\le\epsilon_{1},$ are as defined above. Let $k_{0}>0$
and let $0<\zeta<1.$ Then there exists $\epsilon>0,$ depending only
on $\gamma,$ $\overline{h},$ $\epsilon_{1},$ $k_{0},$ and $\zeta,$
such that the following properties hold for any regular $C^{2}$ curve,
$\alpha(s)=(\tau_{1}(s),\tau_{2}(s))$, $-\ell\le s\le\ell,$ in $F(\epsilon).$ \end{lem}
\begin{enumerate}
\item If $N_{\overline{h}}(s)$ and $N_{e}(s)$ are continuous unit normal
fields along $\alpha(s)$ for the metric $\overline{h}$ and the Euclidean
metric, respectively, then $N_{\overline{h}}(s)$ and $N_{e}(s)$
lie on the same side of the tangent line to $\alpha$ at $s$ when
$s=0$ (and therefore, by continuity, for all $s\in[-\ell,\ell])$
if and only if the same choice of sign replaces $\pm$ in (\ref{eq:Nh})
as in (\ref{eq:Ne}) below.
\item Suppose $N_{\overline{h}}(s)$ and $N_{e}(s)$ are as in part (1),
lying on the same side of the tangent lines to $\alpha.$ Let the
signed curvatures $k(s)$ and $\kappa(s),$ denote, respectively,
the signed curvature of $\alpha(s)$ with respect to $N_{\overline{h}}(s)$
on $(\overline{F}(\epsilon),\overline{h})$ (as defined in (\ref{eq:DefSignedCurvature})
) and the signed curvature of $\alpha(s)$ with respect $N_{e}(s)$
in the Euclidean metric on $\overline{F}(\epsilon)$ (as defined in
(\ref{eq:EuclidCurvature})). If $|k(s)|\ge k_{0}$ for all $s\in[-\ell,\ell],$
then \begin{equation}
1-\zeta<\frac{\kappa(s)}{k(s)}<1+\zeta,\text{\ for\ all\ }s\in[-\ell,\ell].\label{eq:CurvatureRatio}\end{equation}
\end{enumerate}
\begin{proof}
Since the curvatures (in either metric) and the unit normal fields
along $\alpha$ are invariant under reparametrization, we may assume
that $\alpha(s)$ is a unit speed curve in $(\overline{F}(\epsilon),\overline{h}).$
Note that in Fermi coordinates $(\tau_{1},\tau_{2}),$ if $\overline{h}_{ij}=<\partial/\partial\tau_{i},\partial/\partial\tau_{j}>_{\overline{h}}$
and $\Gamma_{ij}^{m}$ are the Christoffel symbols for $\overline{h},$
then \[
\overline{h}_{11}(\tau_{1},0)=1,~~\overline{h}_{12}(\tau_{1},\tau_{2})=0,~~\overline{h}_{22}(\tau_{1},\tau_{2})=1,\]
 \[
\frac{\partial\overline{h}_{11}}{\partial\tau_{2}}(\tau_{1},0)=0,~~\frac{\partial\overline{h}_{12}}{\partial\tau_{2}}(\tau_{1},\tau_{2})=0,~~\frac{\partial\overline{h}_{22}}{\partial\tau_{2}}(\tau_{1},\tau_{2})=0,\]
 \begin{equation}
\overline{h}_{ij}(\tau_{1},\tau_{2})=\delta_{ij}+o(|\tau_{2}|),\label{eq:hijEstimate}\end{equation}
and \begin{equation}
\Gamma_{ij}^{m}(\tau_{1},\tau_{2})=O(|\tau_{2}|),\label{eq:Christoffel}\end{equation}
for $i,j,m\in\{1,2\}.$ Here, and throughout this proof, $O(|\tau_{2}|)$
and $o(|\tau_{2}|)$ denote functions whose absolute values are bounded,
respectively, by constant multiples of $|\tau_{2}|$ and $\tau_{2}^{2},$
where the constants can be chosen independently of $\alpha$ and $\epsilon.$

Since \begin{equation}
\alpha'(s)=\tau_{1}'(t)\frac{\partial}{\partial\tau_{1}}+\tau_{2}'(t)\frac{\partial}{\partial\tau_{2}}\label{eq:AlphaPrime}\end{equation}
and $<\alpha'(s),\alpha'(s)>_{\overline{h}}=1$, we have \begin{equation}
\overline{h}_{11}(\tau_{1}')^{2}+(\tau_{2}')^{2}=1.\label{eq:AlphaLength}\end{equation}
It follows that $|\tau_{1}'|\leq(\overline{h}_{11})^{-1/2}$ and $|\tau_{2}'|\leq1.$
From (\ref{eq:hijEstimate}) and (\ref{eq:AlphaLength}), we obtain
\begin{equation}
(\tau_{1}')^{2}+(\tau_{2}')^{2}=1+o(|\tau_{2}|).\label{eq:CurveLengthEstimate}\end{equation}
Note that \begin{equation}
N_{\overline{h}}(s)=\pm\left[\frac{-\tau_{2}'(s)}{\sqrt{\overline{h}_{11}}}\frac{\partial}{\partial\tau_{1}}+\tau_{1}'(s)\sqrt{\overline{h}_{11}}\frac{\partial}{\partial\tau_{2}}\right],\label{eq:Nh}\end{equation}
 where the same choice of $\pm$ is made for all $s.$ For the Euclidean
normal vector, we have\begin{equation}
N_{e}(s)=\pm\left[\frac{-\tau_{2}'(s)}{\sqrt{(\tau_{1}'(s))^{2}+(\tau_{2}'(s))^{2}}}\frac{\partial}{\partial\tau_{1}}+\frac{\tau_{1}'(s)}{\sqrt{(\tau_{1}'(s))^{2}+(\tau_{2}'(s))^{2}}}\frac{\partial}{\partial\tau_{2}}\right],\label{eq:Ne}\end{equation}
also with the same choice of $\pm$ for all $s.$ If we make the same
choice of $\pm$ in (\ref{eq:Nh}) as in (\ref{eq:Ne}), then it follows
from (\ref{eq:AlphaLength}) and (\ref{eq:hijEstimate}) for $(i,j)=(1,1)$
that \begin{equation}
\measuredangle(N_{\overline{h}}(s),N_{e}(s))=o(|\tau_{2}|),\label{eq:NhNeComparison}\end{equation}
 where $\measuredangle$ denotes the angle measured in the Euclidean
coordinate system. In this case, if $\epsilon$ is sufficiently small
and $|\tau_{2}|<\epsilon,$ then $N_{\overline{h}}(s)$ and $N_{e}(s)$
lie on the same side of the tangent line to $\alpha$ at $s.$ Conversely,
if we had made the opposite choice of signs in (\ref{eq:Nh}) from
that in (\ref{eq:Ne}), then $N_{\overline{h}}(s)$ and $N_{e}(s)$
would lie on opposite sides of the tangent line to $\alpha$ at $s.$
The assertion in part (1) follows.

In the proof of part (2), we may assume that $\pm$ in both (\ref{eq:Nh})
and (\ref{eq:Ne}) is taken to be +. Taking the covariant derivative
$D/ds,$ with respect to the metric $\overline{h},$ of (\ref{eq:AlphaPrime}),
we obtain

\begin{eqnarray}
\frac{D}{ds}\alpha'(s) & = & \sum_{m=1}^{2}\left(\tau_{m}''(s)+\sum_{1\le i,j\le2}\Gamma_{ij}^{m}\tau_{i}'(s)\tau_{j}'(s)\right)\frac{\partial}{\partial\tau_{m}}\nonumber \\
 & = & \sum_{m=1}^{2}\left(\tau_{m}''(s)+O(|\tau_{2}|)\right)\frac{\partial}{\partial\tau_{m}}.\label{eq:Covariant}\end{eqnarray}
Then the signed curvature of $\alpha(s)$ with respect to $N_{\overline{h}}$
in the metric $\overline{h}$ is \begin{eqnarray}
k(s) & = & -\left\langle \frac{D}{ds}\alpha'(s),N_{\overline{h}}(s)\right\rangle _{\overline{h}}\nonumber \\
 & = & (\tau_{1}''\tau_{2}'-\tau_{1}'\tau_{2}'')\sqrt{\overline{h}_{11}}+O(|\tau_{2}|).\label{eq:CurvatureEstimate}\end{eqnarray}
Next we find the signed curvature $\kappa(s)$ of $\alpha(s)$ with
respect to $N_{e}$ in the Euclidean metric. By (\ref{eq:EuclidCurvature})
and (\ref{eq:Ne}), we have\begin{eqnarray}
\kappa(s) & = & \frac{-\langle\alpha''(s),N_{e}(s)\rangle_{e}}{||\alpha'(s)||_{e}^{2}}\nonumber \\
 & = & \frac{\tau_{1}''\tau_{2}'-\tau_{1}'\tau_{2}''}{\big((\tau_{1}')^{2}+(\tau_{2}')^{2}\big)^{3/2}}\label{eq:KappaEstimate}\end{eqnarray}
If $|k(s)|\ge k_{0}$ and $\epsilon$ is sufficiently small, then
$(\ref{eq:CurvatureRatio})$ follows from (\ref{eq:CurveLengthEstimate}),
(\ref{eq:CurvatureEstimate}), and (\ref{eq:KappaEstimate}).
\end{proof}
In Proposition \ref{pro:Asymptotic} below, we will consider the closure
of an $\epsilon_{0}$-tubular neighborhood of $\widehat{\rho,}$ $(\widehat{F}(\epsilon_{0}),\widehat{h}),$
as described in Section \ref{sec:TubularNeighborhoods}. Either component
of $\widehat{F}(\epsilon_{0})\setminus\text{tr}(\widehat{\rho})$
could serve as the region in which the second of the Fermi coordinates
$(\tau_{1},\tau_{2})$ is positive, depending on how the Fermi coordinates
are chosen. For any geodesic $\gamma:[a,b]\to S,$ we let $-\gamma$
denote the geodesic $\gamma$ transversed in the opposite direction:
$(-\gamma)(t)=\gamma(a+b-t),$ for $a\le t\le b.$
\begin{prop}
\label{pro:Asymptotic}If $\rho:[0,L]\to S$ is the closed geodesic
constructed in Section \ref{sec:Construction-of-Metrics} and $x\in S,$
then there exist $v_{+},v_{-}\in T_{x}^{1}S$ such that $\gamma_{v_{+}}(t)$
and $\gamma_{v_{-}}(t)$ are asymptotic to $\rho$ and $-\rho$ ,
respectively, as $t\to\infty.$ If $x\in\cup_{i=1}^{q}\mathcal{D}_{i}$
then $v_{+}$ and $v_{-}$ can be chosen to be approximately in the
radial direction (see Definition \ref{def:ApproxRadial}). In addition,
if we are given a choice of Fermi coordinates $(\tau_{1},\tau_{2})$
on $(\widehat{F}(\epsilon_{0}),\widehat{h}),$ then we can choose
$v_{+}$ and $v_{-}$ so that for every $\epsilon>0$ there exists
$T=T(\epsilon)>0$ such that $\gamma_{v_{+}}|[T,\infty)$ and $\gamma_{v_{-}}|[T,\infty)$
have lifts $\widehat{\gamma}_{v_{+}}|[T,\infty)$ and $\widehat{\gamma}_{v_{-}}|[T,\infty)$,
respectively, to $(\widehat{F}(\epsilon_{0}),\widehat{h})$ such that
\[
0<\tau_{2}(\widehat{\gamma}_{v_{+}}(t))<\epsilon\text{\ and\ }0<\tau_{2}(\widehat{\gamma}_{v_{-}}(t))<\epsilon,\text{\ for\ }t\ge T.\]
\end{prop}
\begin{proof}
Let $x\in S$ and let $\sigma:[-a_{0},a_{0}]\to T_{x}^{1}S$, $a_{0}>0,$
be a regular $C^{1}$ curve (with constant basepoint $x).$ If $x\in\cup_{i=1}^{q}\mathcal{D}_{i},$
then $\sigma$ is chosen so that $\sigma(s)$ is approximately in
the radial direction for all $s\in[-a_{0},a_{0}].$ We will show that
for some $s_{2}\in(-a_{0},a_{0}),$ the vector $\sigma(s_{2})$ satisfies
the conditions required of $v_{+}.$ Let $R$ be as in Lemma \ref{lem:H=00003D0},
and let $t_{0}>R$ be such that $x$ is not conjugate to $\gamma_{\sigma(0)}(t_{0})$
along $\gamma_{\sigma(0)}|[0,t_{0}].$ Let $0<a<a_{0}/2$ and $0<b<(t_{0}-R)/2,$
and assume $a,b$ are sufficiently small so that $\exp_{x}:\{t\sigma(s)\in T_{x}S:-2a<s<2a$,
$t_{0}-2b<t<t_{0}+2b\}$ $\to S$ is a diffeomorphism onto its image.
Define \[
\mathcal{A}_{0}=\mathcal{A}_{0}(a,b,t_{0})=\{\varphi^{t}(\sigma(s)):-a<s<a,t_{0}-b<t<t_{0}+b\}\]
 and \[
\mathcal{A}_{1}=\mathcal{A}_{1}(a,b,t_{0})=\{\varphi^{t}(\sigma(s)):-a/2\le s\le a/2,t_{0}-(b/2)\le t\le t_{0}+(b/2)\}.\]
Note that $\varphi^{t}(\sigma(s))\subset(T^{1}S)\setminus(\cup_{i=1}^{q}\mathcal{Z}_{i}),$
for $s\in[-a_{0},a_{0}]$ and $t\ge0$ (where $\mathcal{Z}_{i}$ is
in Definition \ref{def:Z}), because $\sigma(s)\notin(\cup_{i=1}^{q}\mathcal{Z}_{i})$
and $\varphi^{t}((T^{1}S)\setminus(\cup_{i=1}^{q}\mathcal{Z}_{i}))\subset$
$(T^{1}S)\setminus(\cup_{i=1}^{q}\mathcal{Z}_{i})$ for all $t\ge0.$
We require $a,b$ to be sufficiently small so that there exist an
open subset $W$ of $T^{1}S$ with $\overline{\mathcal{A}_{0}}\subset W\subset(T^{1}S)\setminus(\cup_{i=1}^{q}\mathcal{Z}_{i})$
and a coordinate chart $\Psi:W\to\mathbb{R}^{3}$ such that $\Psi(\mathcal{A}_{0})$
is an open subset of $\mathbb{R}^{2}.$ Here, and in the proof of
Lemma \ref{lem:StableConnecting}, we identify $\mathbb{R}^{2}$ with
$\mathbb{R}^{2}\times\{0\}\subset\mathbb{R}^{3}.$ For $\eta>0,$
let $\mathcal{A}=\mathcal{A}(\eta)$ be an $\eta-$neighborhood of
$\mathcal{A}_{1}.$

We will continue with the proof of Proposition \ref{pro:Asymptotic}
after Lemma \ref{lem:StableConnecting} and Corollary \ref{cor:VisitU}.
\begin{lem}
\label{lem:StableConnecting}If $\eta=\eta(a,b,W,\Psi)$ is sufficiently
small, then every vector $w\in\mathcal{A}(\eta)\setminus\mathcal{A}_{0}$
can be joined to a vector in $\mathcal{A}_{0}$ by a stable curve
with finite Lyapunov length.\end{lem}
\begin{proof}
By Lemma \ref{lem:H=00003D0}, the curves $s\mapsto\varphi^{t}(\sigma(s)),$
$-a_{0}\le s\le a_{0},$ are approximately unstable for $t\ge R.$
Thus each tangent plane to $\overline{\mathcal{A}_{0}}$ is spanned
by a vector tangent to an orbit of the geodesic flow and a vector
in an unstable cone. By Lemma \ref{lem:LineField}, the $E^{s}$ line
field is contained in the interiors of the stable cones at vectors
in $\overline{\mathcal{A}_{0}}.$ Hence Lemma \ref{lem:DisjointCones}
implies that $E^{s}$ is transversal to the tangent plane to $\overline{\mathcal{A}_{0}}$
at each vector in $\overline{\mathcal{A}_{0}}.$ Since $E^{s}$ is
continuous on $W,$ it follows that there exists $c>0$ such that
$\Psi(\overline{\mathcal{A}_{0}})\times(-c,c)\subset\Psi(W)$ and
the line field $d\Psi(E^{s})$ is uniformly transversal to horizontal
planes in $\Psi(\overline{\mathcal{A}_{0}})\times(-c,c).$ Let $\zeta>0$
be such that within $\Psi(\overline{\mathcal{A}_{0}})\times(-c,c)$
any unit vector along $d\Psi(E^{s})$ has a component of absolute
value greater than $\zeta$ in the vertical direction. Let $h_{0}=$
dist$(\Psi(\mathcal{A}_{1}),\Psi(\partial\mathcal{A}_{0})).$ We choose
$\eta$ sufficiently small so that $\mathcal{A}(\eta)\subset W,$
dist$(\Psi(\mathcal{A}(\eta)),\Psi(\partial\mathcal{A}_{0})\times\mathbb{R})>h_{0}/2$,
and for any $(p_{1},p_{2},p_{3})\in\Psi(\mathcal{A}(\eta)),$ we have
$|p_{3}|<$ min$(c,h_{0}\zeta/2).$ Let $w\in\mathcal{A}(\eta)\setminus\mathcal{A}_{0}$.
Then $\Psi(w)=(p_{1},p_{2},p_{3}),$ where $p_{3}\ne0.$ We suppose
$p_{3}>0.$ Since $d\Psi(E_{s})$ is a continuous line field on $d\Psi(W),$
there exists a unit speed curve $\tilde{\beta}$ that is everywhere
tangent to $d\Psi(E^{s})$ and starts at $\tilde{\beta}(0)=p_{3}$
with $\tilde{\beta}'(0)$ having a negative component in the vertical
direction. Then $\tilde{\beta}$ must reach $\mathbb{R}^{2}\times\{0\}$
in time less than $p_{3}/\zeta,$ unless it first exits the region
$\Psi(\mathcal{A}_{0})\times\mathbb{R}.$ But it cannot exit this
region in time less than $h_{0}/2,$ and by our choice of $\eta,$
we have $p_{3}/\zeta<h_{0}/2.$ Thus, for some $\ell_{1}\in(0,p_{3}/\zeta),$
the curve $\tilde{\beta}(t),$ $0\le t\le\ell_{1},$ connects $(p_{1},p_{2},p_{3})$
to a point in $\Psi(\mathcal{A}_{0}).$ Then $\beta(t)=\Psi^{-1}(\tilde{\beta}(t)),$
$0\le t\le\ell_{1},$ is a stable curve that joins $w$ to a vector
in $\mathcal{A}_{0}.$ Since tr$(\beta)$ is a compact subset of $(T^{1}S)\setminus(\cup_{i=1}^{q}\mathcal{Z}_{i}),$
it follows from Remark \ref{rem:LyapunovLengthBound} that $\beta$
has finite Lyapunov length. The case $p_{3}<0$ can be handled similarly. \end{proof}
\begin{cor}
\label{cor:VisitU}If $\mathcal{U}_{0}$ is a nonempty open subset
of $T^{1}\mathcal{N}_{2},$ where $\mathcal{N}_{2}$ is as in Lemma
\ref{lem:StableCurveContraction}, then there exists $s_{0}\in(-a,a)$
and a sequence $(t_{n})_{n=1,2,\dots}$ with $t_{1}\ge0$ and $t_{n+1}>1+t_{n}$
for $n=1,2,\dots,$ such that $\varphi^{t_{n}}(\sigma(s_{0}))\in\mathcal{U}_{0},$
for $n=1,2,\dots.$\end{cor}
\begin{proof}
Let $\mathcal{U}_{1}$ be a nonempty open subset of $\mathcal{U}_{0}$
such that $\overline{\mathcal{U}_{1}}\subset\mathcal{U}_{0},$ and
let $\eta>0$ be such that Lemma \ref{lem:StableConnecting} holds
for $\mathcal{A}=\mathcal{A}(\eta).$ Then $\mathcal{U}_{1}$ and
$\mathcal{A}$ each have positive Liouville measure. By the ergodicity
of the geodesic flow $\varphi^{t},$ we know that there exists $w\in\mathcal{A}$
 and a sequence $(\widetilde{t}_{n})_{n=1,2,\dots}$such that $\widetilde{t}_{1}\ge0$
and $\widetilde{t}_{n+1}>1+\widetilde{t}_{n}$ for $n=1,2,\dots,$
and $\varphi^{\tilde{t}_{n}}(w)\in\mathcal{U}_{1}$ for $n=1,2,\dots.$
If $w\notin\mathcal{A}_{0}$ then by Lemma \ref{lem:StableConnecting},
there is a stable curve $\beta$ from $w$ to a vector $v$ in $\mathcal{A}_{0}.$
The vector $v$ can be written as $v=\varphi^{\tilde{t}_{0}}(\sigma(s_{0}))$
for some $s_{0}\in(-a,a)$ and $t_{0}-b<\widetilde{t}_{0}<t_{0}+b.$
By part (2) of Lemma \ref{lem:StableCurveContraction}, $\lim_{n\to\infty}\mathcal{L}(\varphi^{\tilde{t}_{n}}(\beta))=0.$
Thus there exists $N$ such that $\widetilde{t}_{N+1}>\widetilde{t}_{0},$
and for $n>N,$ $\varphi^{\tilde{t}_{n}}(v)\in\mathcal{U}_{0}.$ Then
the Corollary holds with $t_{n}=\widetilde{t}_{N+n}+\widetilde{t}_{0}.$
If $w\in\mathcal{A}_{0},$ then we already have $w=\varphi^{\tilde{t}_{0}}(\sigma(s_{0}))$
for some $s_{0}\in(-a,a)$ and $t_{0}-b<\widetilde{t}_{0}<t_{0}+b$
and the Corollary holds with $t_{n}=\widetilde{t}_{n}+\widetilde{t}_{0}.$
\end{proof}
We now complete the proof of Proposition \ref{pro:Asymptotic}. Let
$\epsilon_{0}\in(0,1)$, $F(\epsilon_{0}),$ and $\widehat{F}(\epsilon_{0})$
be as described in Section \ref{sec:TubularNeighborhoods}, with Fermi
coordinates $(\tau_{1},\tau_{2})$ along $\widehat{\rho}$ in the
metric $\widehat{h}$ defined in $(\mathbb{R}/L\mathbb{Z})\times[-\epsilon_{0},\epsilon_{0}].$
Let $\widehat{Z}$ be the unit vector field on $\widehat{F}(\epsilon_{0})$
that is asymptotic to $\widehat{\rho},$ as in Section \ref{sec:TubularNeighborhoods}.

Since $F(\epsilon_{0})\subset S\setminus(\cup_{i=1}^{q}\mathcal{D}_{i}),$
it follows from (\ref{eq:StrictUnstableInvariance}) that if $v\in T^{1}S$
has its basepoint in $F(\epsilon_{0})$ and $d\varphi^{-1}v\in\mathcal{K}_{\varphi^{-1}v}^{u},$
then $d\varphi^{1}(\mathcal{K}_{d\varphi^{-1}v}^{u})\subset\interior\mathcal{K}_{v}^{u}=\interior\mathcal{K}_{v}^{+}.$
Moreover, by a compactness argument, there exist $k_{0},k_{1},$ $0<k_{0}<k_{1}<\infty,$
such that for all such $v,$ $d\varphi^{1}(\mathcal{K}_{d\varphi^{-1}v}^{u})$
is contained in a cone in $\mathcal{K}_{v}^{+}$ that is bounded by
lines of slopes $k_{0}$ and $k_{1}$ in the $H,V$ coordinates. Therefore,
if $t\ge R+1$ and $\sigma_{0}$ is the restriction of $\sigma$ to
a subinterval of $[-a_{0},a_{0}]$ such that $\varphi^{t}\sigma_{0}$
has all of its basepoints in $F(\epsilon_{0}),$ then the curvature
of the curve of basepoints of $\varphi^{t}\sigma_{0}$ (with respect
to the metric $h$ and the normal field given by $\varphi^{t}\sigma_{0}$)
is in the interval $[k_{0},k_{1}],$ because $\varphi^{t}\sigma_{0}=\varphi^{1}(\varphi^{t-1}\sigma_{0})$
and $\varphi^{t-1}\sigma_{0}$ is approximately unstable (by Remark
\ref{rem:ConstantBasepointCurve}). Likewise, the curve of basepoints
of any lift of such a $\varphi^{t}\sigma_{0}$ to $T^{1}(\widehat{F}(\epsilon_{0}))$
has curvature (with respect to the metric $\widehat{h})$ in $[k_{0},k_{1}].$

We now apply Lemma \ref{lem:FermiCurvature} with $\gamma=\widehat{\rho},$
$\epsilon_{1}=\epsilon_{0},$ $\overline{F}(\epsilon_{0})=\widehat{F}(\epsilon_{0}),$
and $\overline{h}=\widehat{h}.$ Let $\epsilon\in(0,\epsilon_{0}]$
be such that the conclusion of Lemma \ref{lem:FermiCurvature} holds
for $\zeta=1/2$ and $k_{0}$ as above. Let $\kappa_{0}=k_{0}/2$
and $\kappa_{1}=(3/2)k_{1},$ and let $\widehat{F}(\epsilon)\subset\widehat{F}(\epsilon_{0})$
be defined as in (\ref{eq:FHat}). By (\ref{eq:NhNeComparison}),
there is a constant $C_{0}>0$ such that along any regular $C^{2}$
curve $\widehat{\alpha}(s)=(\tau_{1}(s),\tau_{2}(s))$, $0\le s\le\ell,$
in $\widehat{F}(\epsilon),$ the unit normal vectors $N_{\widehat{h}}$
and $N_{e}$ to $\widehat{\alpha}$ (in the $\widehat{h}$ metric
and the Euclidean metric, respectively) satisfy \begin{equation}
\measuredangle(N_{\widehat{h}},N_{e})\le C_{0}\tau_{2}^{2},\label{eq:NgNeAngle}\end{equation}
where $\measuredangle$ denotes the angle measured in the Euclidean
coordinate system. Since, by (\ref{eq:hijEstimate}), the ratio of
the Euclidean length to the length in the $\widehat{h}$ metric is
close to $1$ for $\tau_{2}$ close to $0,$ we may also assume that
$\epsilon$ is sufficiently small so that \begin{equation}
||v||_{\widehat{h}}\le2||v||_{e},\label{eq:LengthDoubling}\end{equation}
 for all $v\in TS$ with basepoint in $\widehat{F}(\epsilon).$ After
choosing $\epsilon,$ we choose $\delta$ so that \[
0<\delta<\frac{1}{C_{0}+2}\min\left(\epsilon,\frac{\epsilon\kappa_{0}}{2},\frac{\pi\kappa_{0}}{2\kappa_{1}},\frac{r_{0}\kappa_{0}}{2\sqrt{1+k_{1}^{2}}}\right),\]
 where $r_{0}$ is as in Corollary \ref{cor:UnstableCurveExpansion}.

Let $\widehat{\mathcal{U}}_{0}=\{\widehat{w}\in T_{p}^{1}(\widehat{F}(\epsilon)):p=(\tau_{1},\tau_{2}),$
$\text{\ }$$0<\tau_{2}<\delta,$$\text{\ }$$\measuredangle(\widehat{Z}(p),\partial/\partial\tau_{1})<\delta,$
$\measuredangle(\widehat{w},\widehat{Z}(p))<\delta,\text{\ }\measuredangle(\widehat{w},-\partial/\partial\tau_{2})<\pi/2\text{,\ and\ }\measuredangle(\widehat{w},\partial/\partial\tau_{1})>\measuredangle(\widehat{Z}(p),\partial/\partial\tau_{1})\}.$
For the rest of the proof of this proposition, the signed Euclidean
angle from one vector to another at the same basepoint in $\widehat{F}(\epsilon)$
will be taken to be in $(-\pi,\pi]$ and the counterclockwise direction
will be the positive direction. In particular, for $\widehat{w}\in\widehat{\mathcal{U}}_{0}\cap T_{p}^{1}(\widehat{F}(\epsilon))$
the signed Euclidean angle from $\widehat{Z}(p)$ to $\widehat{w}$
is negative. Let $\mathcal{U}_{0}$ be the image of $\widehat{\mathcal{U}}_{0}$
under the projection $\widehat{\pi}$ from $T^{1}(\widehat{F}(\epsilon))$
to $T^{1}(F(\epsilon)).$ Let $s_{0}\in(-a,a)$ and the sequence $(t_{n})_{n=1,2,\dots,}$
be as in Corollary \ref{cor:VisitU} applied to this choice of $\mathcal{U}_{0}.$
If $n$ is sufficiently large, then by Corollary \ref{cor:UnstableCurveExpansion},
$\mathcal{L}(\varphi^{t_{n}}(\sigma|[s_{0},a)))>r_{0}$ and $\mathcal{L}(\varphi^{t_{n}}(\sigma|(-a,s_{0}]))>r_{0}$.
Now fix a choice of such an $n,$ where we also require that $t_{n}\ge R+1.$

Let $w=\varphi^{t_{n}}(\sigma(s_{0}))$ and let $\widehat{w}\in\widehat{\mathcal{U}}_{0}$
be such that $\widehat{w}$ projects to $w.$ Let $\widehat{\sigma}$
be the lift of $\varphi^{t_{n}}(\sigma(s)),$ $s_{0}\le s\le a_{1},$
to $T^{1}(\widehat{F}(\epsilon)),$ where $\varphi^{t_{n}}(\sigma(s_{0}))$
lifts to $\widehat{w}$ and the curve $\widehat{\sigma}$ is truncated,
if necessary, at $s=a_{1},$ where $\widehat{\sigma}$ exits $T^{1}(\widehat{F}(\epsilon))$.
(If it doesn't exit $T^{1}(\widehat{F}(\epsilon)),$ we take $a_{1}=a.)$
Let $\widehat{\alpha}(s),$ $s_{0}\le s\le a_{1},$ be the curve of
basepoints of $\widehat{\sigma}.$ Then $\widehat{\alpha}'(s_{0})$
is orthogonal to $\widehat{w}$ in the $\widehat{h}$ metric. Let
$N_{\widehat{h}}$ and $N_{e}$ be unit normal fields along $\widehat{\alpha}(s)$
in the $\widehat{h}$-metric and the Euclidean metric, chosen so that
$N_{\widehat{h}}(0)=\widehat{w}$ and $N_{\widehat{h}}(s)$ and $N_{e}(s)$
lie on the same side of the tangent line to $\widehat{\alpha}$ at
all $s\in[0,a_{1}].$ Since $\measuredangle(\widehat{w},\partial/\partial\tau_{1})<2\delta<\pi/2$
and $\measuredangle(\widehat{w},-\partial/\partial\tau_{2})<\pi/2,$
$N_{\widehat{h}}(0)$ has a positive component in the $\partial/\partial\tau_{1}$
direction and a negative component in the $\partial/\partial\tau_{2}$
direction, which, according to equations (\ref{eq:Nh}) and (\ref{eq:Ne}),
implies that the same is true of $N_{e}(0).$ In particular, we see
that $\widehat{\alpha}'(s_{0})$ must have a nonzero component in
the $\partial/\partial\tau_{2}$ direction. We will assume that this
component is positive. (If not, we would replace $\varphi^{t_{n}}(\sigma|[s_{0},a])$
by $\varphi^{t_{n}}(\sigma|[-a,s_{0}])$ in our argument.) We may
assume that the parametrization of $\sigma$ is such that $s_{0}=0$
and $\widehat{\alpha}(s)$ is parametrized by Euclidean arc length.

\begin{figure}[htbp]\begin{center}
\begin{picture}(0,0)%
\includegraphics{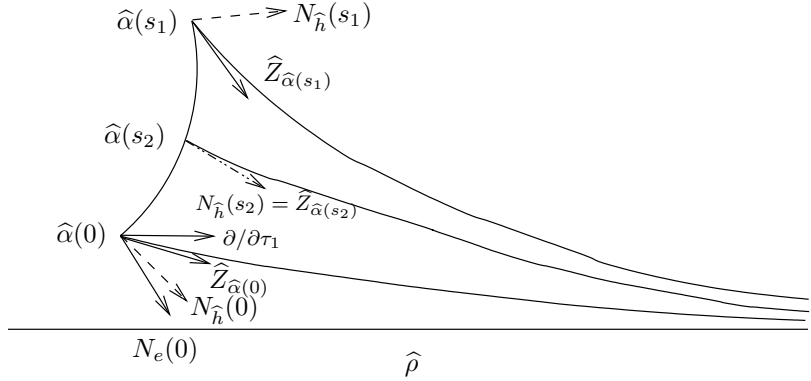}%
\end{picture}%
\setlength{\unitlength}{3947sp}%
\begingroup\makeatletter\ifx\SetFigFont\undefined%
\gdef\SetFigFont#1#2#3#4#5{%
  \reset@font\fontsize{#1}{#2pt}%
  \fontfamily{#3}\fontseries{#4}\fontshape{#5}%
  \selectfont}%
\fi\endgroup%
\begin{picture}(6154,2434)(-1064,-7130) \put(827,-6973){\makebox(0,0)[lb]{\smash{{\SetFigFont{10}{14.4}{\familydefault}{\mddefault}{\updefault}{$N_e(0)$}%
}}}} \put(1834,-4867){\makebox(0,0)[lb]{\smash{{\SetFigFont{10}{14.4}{\familydefault}{\mddefault}{\updefault}{$N_{\widehat{h}}(s_1)$}%
}}}} \put(730,-4914){\makebox(0,0)[lb]{\smash{{\SetFigFont{10}{14.4}{\familydefault}{\mddefault}{\updefault}{$\widehat{\alpha}(s_1)$}%
}}}} \put(650,-5630){\makebox(0,0)[lb]{\smash{{\SetFigFont{10}{14.4}{\familydefault}{\mddefault}{\updefault}{$\widehat{\alpha}(s_2)$}%
}}}} \put(350,-6250){\makebox(0,0)[lb]{\smash{{\SetFigFont{10}{14.4}{\familydefault}{\mddefault}{\updefault}{$\widehat{\alpha}(0)$}%
}}}} \put(1649,-5240){\makebox(0,0)[lb]{\smash{{\SetFigFont{10}{14.4}{\familydefault}{\mddefault}{\updefault}{$\widehat{Z}_{\widehat{\alpha}(s_1)}$}%
}}}} \put(2538,-7057){\makebox(0,0)[lb]{\smash{{\SetFigFont{10}{14.4}{\familydefault}{\mddefault}{\updefault}{$\widehat{\rho}$}%
}}}} \put(1397,-6269){\makebox(0,0)[lb]{\smash{{\SetFigFont{8}{14.4}{\familydefault}{\mddefault}{\updefault}{$\partial/\partial\tau_1$}%
}}}} \put(1224,-6050){\makebox(0,0)[lb]{\smash{{\SetFigFont{8}{14.4}{\familydefault}{\mddefault}{\updefault}{$N_{\widehat{h}}(s_2)=\widehat{Z}_{\widehat{\alpha}(s_2)}$}%
}}}} \put(1311,-6541){\makebox(0,0)[lb]{\smash{{\SetFigFont{10}{14.4}{\familydefault}{\mddefault}{\updefault}{$\widehat{Z}_{\widehat{\alpha}(0)}$}%
}}}} \put(1190,-6706){\makebox(0,0)[lb]{\smash{{\SetFigFont{10}{14.4}{\familydefault}{\mddefault}{\updefault}{$N_{\widehat{h}}(0)$}%
}}}} \end{picture}%
\caption{Rotation of $N_{\widehat{h}}$.} \label{figure:2} \end{center} \end{figure}

By the choice of $\epsilon$ and by (\ref{eq:NgNeAngle}), we know
that $\measuredangle(N_{e}(0),N_{\widehat{h}}(0))<C_{0}\delta^{2}<C_{0}\delta.$
Thus $\measuredangle(N_{e}(0),\partial/\partial\tau_{1})<(C_{0}+2)\delta.$
According to Lemma \ref{lem:FermiCurvature}, the Euclidean curvature
of $\widehat{\alpha}(s)$ is between $\kappa_{0}$ and $\kappa_{1}$
for $0\le s\le a_{1}.$ Thus $N_{e}(s)$ rotates in the counterclockwise
direction at a rate between $\kappa_{0}$ and $\kappa_{1}$ radians
per unit time. For $s\in[0,\min(a_{1},(C_{0}+2)\delta/\kappa_{0})],$
the signed Euclidean angle from $\partial/\partial\tau_{1}$ to $N_{e}(s),$
is strictly between $-(C_{0}+2)\delta+\kappa_{0}s$ and $\kappa_{1}s,$
which implies it is strictly between $-(C_{0}+2)\delta$ and $\pi/2.$
Thus $\widehat{\alpha}'(s)$ has a non-zero component in the $\partial/\partial\tau_{2}$
direction for $s\in[0,\min(a_{1},(C_{0}+2)\delta/\kappa_{0})].$ In
fact, this component must be positive, because $\widehat{\alpha}'(0)$
has a positive component in the $\partial/\partial\tau_{2}$ direction.
Since the component of $\widehat{\alpha}'(s)$ in the $\partial/\partial\tau_{2}$
direction is at most $1,$ and $0<\tau_{2}(\widehat{\alpha}(0))<\delta<\epsilon/2,$
we obtain $0<\tau_{2}(\widehat{\alpha}(s))<\epsilon/2+(C_{0}+2)\delta/\kappa_{0}<\epsilon$
for $s\in[0,\min(a_{1},(C_{0}+2)\delta/\kappa_{0})].$

The length of $\widehat{\sigma}(s),$ $0\le s\le\min(a_{1},(C_{0}+2)\delta/\kappa_{0})]$,
in $T^{1}S$ with the metric induced by $\widehat{h},$ is at most
$(1+k_{1}^{2})^{1/2}$ times the $\widehat{h}$-length of $\widehat{\alpha}(s),$
$0\le s\le\min(a_{1},(C_{0}+2)\delta/\kappa_{0}).$ By (\ref{eq:LengthDoubling}),
the $\widehat{h}$-length of $\widehat{\alpha}(s),$ $0\le s\le\min(a_{1},(C_{0}+2)\delta/\kappa_{0}),$
is at most $2(C_{0}+2)\delta/\kappa_{0}.$ Thus the length of $\widehat{\sigma}(s),$
$0\le s\le\min(a_{1},(C_{0}+2)\delta/\kappa_{0})$, is at most $2(1+k_{1}^{2})^{1/2}(C_{0}+2)\delta/\kappa_{0}$,
which is less than $r_{0}.$ Thus $\widehat{\sigma}(s)$ can neither
exit $T^{1}(\widehat{F}(\epsilon))$ nor reach length $r_{0}$ by
time $s=\min(a_{1},(C_{0}+2)\delta/\kappa_{0}).$ Hence $a_{1}\ge(C_{0}+2)\delta/\kappa_{0}.$

Since $N_{e}(s)$ rotates counterclockwise at a rate of at least $\kappa_{0}$
radians per unit time, there is an $s_{1}\in(0,(C_{0}+2)\delta/\kappa_{0})$
at which $N_{e}(s_{1})$ has a positive component in the $\partial/\partial\tau_{2}$
direction, and by (\ref{eq:Nh}) and (\ref{eq:Ne}), $N_{\widehat{h}}(s_{1})$
also has a positive component in the $\partial/\partial\tau_{2}$
direction. Moreover, $s_{1}$ may be chosen so that $N_{e}(s)$ and
$N_{\widehat{h}}(s)$ have positive components in the $\partial/\partial\tau_{1}$
direction for all $s\in[0,s_{1}].$ Each vector in $\widehat{Z}$
with basepoint in the $\tau_{2}>0$ region has a positive component
in the $\partial/\partial\tau_{1}$ direction and a negative component
in the $\partial/\partial\tau_{2}$ direction. Thus the signed Euclidean
angle from $\widehat{Z}_{\widehat{\alpha}(s)}$ to $N_{\widehat{h}}(s)$
changes from being negative at $s=0$ to being positive at $s=s_{1}.$
Both $N_{\widehat{h}}(s)$ and $\widehat{Z}_{\widehat{\alpha}(s)}$
are continuous unit vector fields along $\widehat{\alpha}(s),$ $0\le s\le s_{1}.$
By the intermediate value theorem, there exists $s_{2}\in(0,s_{1})$
such that $N_{\widehat{h}}(s_{2})=\widehat{Z}_{\widehat{\alpha}(s_{2})}.$
Since $N_{\widehat{h}}(s_{2})$ is a lift to $T^{1}(\widehat{F}(\epsilon))$
of $\varphi^{t_{n}}(\sigma(s_{2})),$ it follows that $v_{+}=\sigma(s_{2})$
has the required properties. The existence of $v_{-}$ also follows,
since we may replace $\rho$ by $-\rho$ in the above proof. \end{proof}
\begin{prop}
\label{pro:BangertGutkinProperty}Let $\rho$ be the closed geodesic
constructed in Section \ref{sec:Construction-of-Metrics}. For each
$(x,y)\in S\times S$ there exists an infinite family of distinct
geodesics $\gamma_{n}:[0,L_{n}]\to S,$ $n=1,2,\dots,$ from $x$
to $y$ with $\lim_{n\to\infty}L_{n}=\infty$ satisfying the following:
for every $\epsilon>0,$ there exists $T=T(\epsilon)>0$ and $N=N(\epsilon)$
such that for $n>N$ and $t\in[T,L_{n}-T],$ we have ${\rm dist}(\gamma_{n}(t),{\rm tr}(\rho))<\epsilon.$
Suppose $\widetilde{\rho}$ and $(\widetilde{F}(\epsilon_{0}),\widetilde{h})$
are as in Section \ref{sec:TubularNeighborhoods}, and we are given
a choice of Fermi coordinates $(\tau_{1},\tau_{2})$ on $(\widetilde{F}(\epsilon_{0}),\widetilde{h})$
such that $\tau_{1}$ is the coordinate along $\widetilde{\rho}$
and $\tau_{2}$ is the coordinate along geodesics perpendicular to
$\widetilde{\rho}.$ Then the geodesics $\gamma_{n}$ can be chosen
so that if $0<\epsilon<\epsilon_{0}$ and $n>N,$ there exists a lift
$\widetilde{\gamma}_{n}|[T,L_{n}-T]$ of $\gamma_{n}|[T,L_{n}-T]$
to $\widetilde{F}(\epsilon_{0})$ such that \[
0<\tau_{2}(\widetilde{\gamma}_{n}(t))<\epsilon,\text{\ for\ }t\in[T,L_{n}-T].\]
\end{prop}
\begin{proof}
Let $x,y\in S,$ and for $0<\epsilon\le\epsilon_{0},$ let $F(\epsilon),$
$\widehat{F}(\epsilon),$ $\widetilde{F}(\epsilon)$ be the closures
of the $\epsilon$-tubular neighborhoods of $\rho,$ $\widehat{\rho},$
$\widetilde{\rho},$ respectively, as defined in Section \ref{sec:TubularNeighborhoods}.
Suppose we are given a choice of Fermi coordinates $(\tau_{1},\tau_{2})$
along $\widetilde{\rho}$ on $\widetilde{F}(\epsilon_{0}).$ Let $(\epsilon_{n})_{n=1,2,\dots}$
be any sequence with $\epsilon_{n}\downarrow0$ and $\epsilon_{1}<\epsilon_{0}.$
We will construct a sequence of geodesics $\gamma_{n}:[0,L_{n}]\to S$,
$n=1,2,\dots,$ from $x$ to $y$ and sequences $(T_{n})_{n=1,2,\dots}$
with $0<T_{n}<L_{n}/2,$ $T_{n}\uparrow\infty,$ such that\[
\text{dist}(\gamma_{n}(t),\text{\ tr}(\rho))<\epsilon_{m},\text{\ for\ }t\in[T_{m},L_{n}-T_{m}]\text{\ and\ }n\ge m\ge1.\]
 Moreover, we will show that there is a lift $\widetilde{\gamma}_{n}$
of $\gamma_{n}|[T_{1},L_{n}-T_{1}]$ to $\widetilde{F}(\epsilon_{0})$
such that \begin{equation}
0<\tau_{2}(\widetilde{\gamma}_{n}(t))<\epsilon_{m},\text{\ for\ }t\in[T_{m},L_{n}-T_{m}]\text{\ and\ }n\ge m\ge1.\label{eq:nmBGproperty}\end{equation}
 This will imply the conclusion of the proposition, because for any
given $\epsilon>0,$ we may choose $m$ such that $\epsilon_{m}<\epsilon,$
and let $N(\epsilon)=m$ and $T(\epsilon)=T_{m}.$

By Proposition \ref{pro:Asymptotic}, there exist vectors $v_{x}\in T_{x}^{1}S,$
$v_{y}\in T_{y}^{1}S$ such that $\gamma_{v_{x}}$ and $\gamma_{v_{y}}$
are asymptotic to $\rho$ and $-\rho,$ respectively, and for $T$
sufficiently large, the second Fermi coordinate of the lifts of $\gamma_{v_{x}}|[T,\infty)$
and $\gamma_{v_{y}}|[T,\infty)$ to $\widetilde{F}(\epsilon_{0})$
is always positive. Let $\sigma_{x}:[0,a_{0}]\to T_{x}^{1}S$ and
$\sigma_{y}:[0,b_{0}]\to T_{y}^{1}S,$ for some $a_{0},b_{0}>0,$
be one-to-one regular $C^{1}$ curves with $\sigma_{x}(0)=v_{x},$
$\sigma_{y}(0)=v_{y},$ such that for all sufficiently large $t,$
the derivative with respect to $s$ at $s=0$ of the curve of basepoints
of $\varphi^{t}(\sigma_{x}(s))$ has a positive component in the $\partial/\partial\tau_{2}$
direction, and similarly for $\varphi^{t}(\sigma_{y}(s)).$ If $x\in\cup_{i=1}^{q}\mathcal{D}_{i},$
then $v_{x}$ and the curve $\sigma_{x}$ can be chosen so that $\sigma_{x}(s)$
is approximately in the radial direction for all $s\in[0,a_{0}],$
and similarly if $y\in\cup_{i=1}^{q}\mathcal{D}_{i}.$ Thus the curves
$\varphi^{t}\sigma_{x}$ and $\varphi^{t}\sigma_{y}$ are approximately
unstable for $t\ge R,$ where $R$ is as in Lemma \ref{lem:H=00003D0}.

\begin{figure}[htbp]\begin{center}
\begin{picture}(0,0)%
\includegraphics{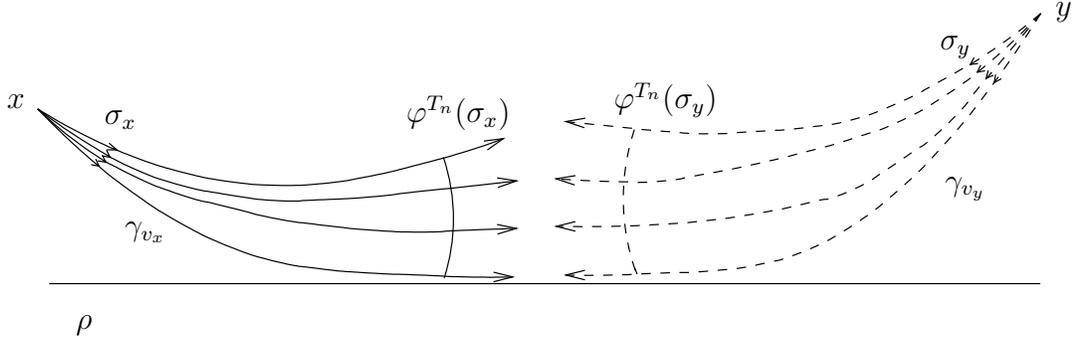}%
\end{picture}%
\setlength{\unitlength}{3947sp}
\begingroup\makeatletter\ifx\SetFigFont\undefined%
\gdef\SetFigFont#1#2#3#4#5{%
  \reset@font\fontsize{#1}{#2pt}%
  \fontfamily{#3}\fontseries{#4}\fontshape{#5}%
  \selectfont}%
\fi\endgroup%
\begin{picture}(6676,2310)(-265,-1804) \put(3622,-358){\makebox(0,0)[lb]{\smash{{\SetFigFont{12}{14.4}{\familydefault}{\mddefault}{\updefault}{$\varphi^{T_n}(\sigma_y)$}%
}}}}

\put(-190,-356){\makebox(0,0)[lb]{\smash{{\SetFigFont{12}{14.4}{\familydefault}{\mddefault}{\updefault}{$x$}%
}}}} \put(424,-458){\makebox(0,0)[lb]{\smash{{\SetFigFont{12}{14.4}{\familydefault}{\mddefault}{\updefault}{$\sigma_x$}%
}}}} \put(550,-1153){\makebox(0,0)[lb]{\smash{{\SetFigFont{12}{14.4}{\familydefault}{\mddefault}{\updefault}{$\gamma_{v_x}$}%
}}}} \put(244,-1740){\makebox(0,0)[lb]{\smash{{\SetFigFont{12}{14.4}{\familydefault}{\mddefault}{\updefault}{$\rho$}%
}}}} \put(2317,-441){\makebox(0,0)[lb]{\smash{{\SetFigFont{12}{14.4}{\familydefault}{\mddefault}{\updefault}{$\varphi^{T_n}(\sigma_x)$}%
}}}} \put(6396,229){\makebox(0,0)[lb]{\smash{{\SetFigFont{12}{14.4}{\familydefault}{\mddefault}{\updefault}{$y$}%
}}}} \put(5670,15){\makebox(0,0)[lb]{\smash{{\SetFigFont{12}{14.4}{\familydefault}{\mddefault}{\updefault}{$\sigma_y$}%
}}}} \put(5702,-871){\makebox(0,0)[lb]{\smash{{\SetFigFont{12}{14.4}{\familydefault}{\mddefault}{\updefault}{$\gamma_{v_y}$}%
}}}} \end{picture}%
\caption{Matching a geodesic starting at $x$ with a geodesic starting at $y$.}
\label{figure:1} \end{center} \end{figure}

Let $T_{0}=R,$ and let $a_{0}$ and $b_{0}$ be as above. We will
choose $T_{n},$ $\overline{t}_{x,n},$ $\overline{t}_{y,n},$ $a_{n},$
and $b_{n}$ inductively so that for $n=1,2,\dots,$ we have the following:
$T_{n}>T_{n-1},$ $T_{n+1}>\max(\overline{t}_{x,n},\overline{t}_{y,n}),$
$0<a_{n}<a_{n-1},$ and $0<b_{n}<b_{n-1}.$ Further conditions on
these parameters will be imposed below. We require $T_{n}$ to be
sufficiently large so that $\gamma_{v_{x}}|[T_{n},\infty),$ $\gamma_{v_{y}}|[T_{n},\infty)$
have lifts $\widetilde{\gamma}_{x,n},$ $\widetilde{\gamma}_{y,n},$
respectively, in $\widetilde{F}(\epsilon_{0})$ so that dist$(\widetilde{\gamma}_{x,n}(t),\text{tr}(\widetilde{\rho}))<\epsilon_{n}/2$
and dist$(\widetilde{\gamma}_{y,n}(t),\text{tr}(\widetilde{\rho}))<\epsilon_{n}/2$
for $t\ge T_{n}.$ We let $\widetilde{\sigma}_{x,n},$ $\widetilde{\sigma}_{y,n}$
be lifts to $\widetilde{F}(\epsilon_{0})$ of $\varphi^{T_{n}}(\sigma_{x}|[0,a_{n}]),$
$\varphi^{T_{n}}(\sigma_{y}|[0,b_{n}])$, respectively, where $a_{n},$
$b_{n}$ are chosen so that the curve of basepoints of $\widetilde{\sigma}_{x,n}(s),$
for $0\le s\le a_{n},$ and the curve of basepoints of $\widetilde{\sigma}_{y,n}(s),$
for $0\le s\le b_{n},$ have length less than $\epsilon_{n}/2.$ Later
we will impose an additional condition relating the choices of the
lift $\widetilde{\sigma}_{x,n}$ and the lift $\widetilde{\sigma}_{y,n}.$
We let $\widetilde{\gamma}_{x,n}$ and $\widetilde{\gamma}_{y,n}$
be lifts of $\gamma_{v_{x}}|[T_{n},\infty)$ and $\gamma_{v_{y}}|[T_{n},\infty)$
chosen so that $\widetilde{\gamma}_{x,n}'(T_{n})=\widetilde{\sigma}_{x,n}(0)$
and $\widetilde{\gamma}_{y,n}'(T_{n})=\widetilde{\sigma}_{y,n}(0).$
We will show that there exist $\bar{t}_{x,n},\bar{t}_{y,n}>T_{n}$
and $\bar{a}_{n}\in(0,a_{n}),$ $\bar{b}_{n}\in(0,b_{n})$ such that
$\gamma_{\sigma_{x}(\bar{a}_{n})}|[0,\bar{t}_{x,n}]$ joins smoothly
to $-\big(\gamma_{\sigma_{y}(\bar{b}_{n})}|[0,\bar{t}_{y,n}]\big)$
at $\gamma_{\sigma_{x}(\bar{a}_{n})}(\bar{t}_{x,n})=\gamma_{\sigma_{y}(\bar{b}_{n})}(\bar{t}_{y,n}),$
to form a geodesic $\gamma_{n}$ from $x$ to $y$ of length $L_{n}>2T_{n}.$
Our construction will be such that dist$(\gamma_{n}(t),\text{tr}(\widetilde{\rho}))<\epsilon_{m}$
for $t\in[T_{m},L_{n}-T_{m}],$ for $1\le m\le n.$

\begin{figure}[htbp]\begin{center}
\begin{picture}(0,0)%
\includegraphics{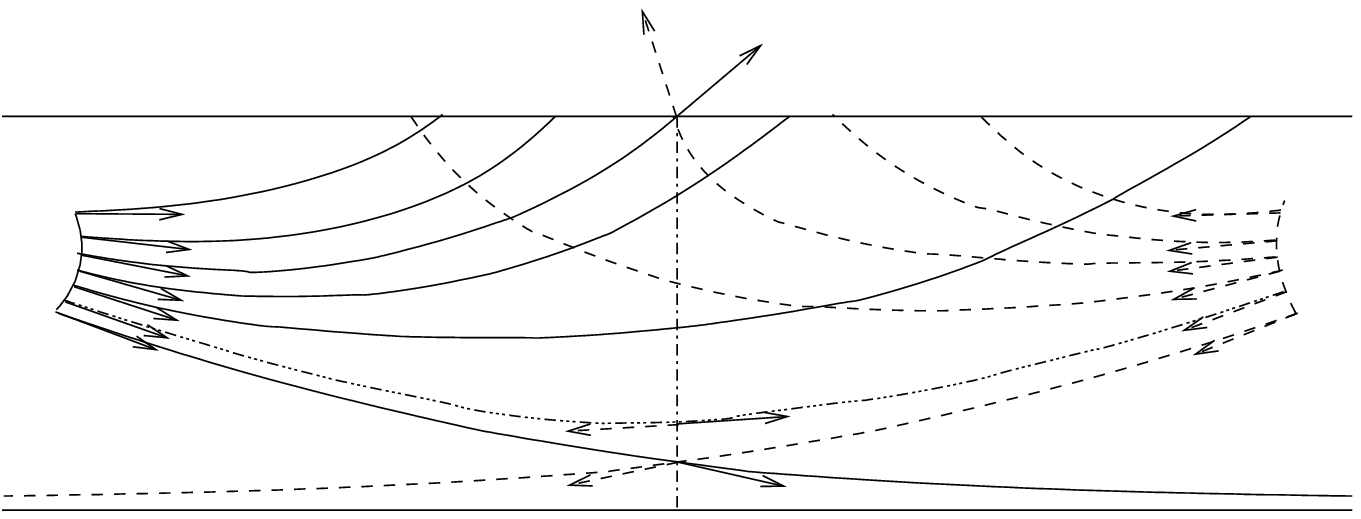}%
\end{picture}%
\setlength{\unitlength}{3947sp}%
\begingroup\makeatletter\ifx\SetFigFont\undefined%
\gdef\SetFigFont#1#2#3#4#5{%
  \reset@font\fontsize{#1}{#2pt}%
  \fontfamily{#3}\fontseries{#4}\fontshape{#5}%
  \selectfont}%
\fi\endgroup%
\begin{picture}(10124,3129)(-511,-4092) \put(2050,-3200){\makebox(0,0)[lb]{\smash{{\SetFigFont{10}{14.4}{\familydefault}{\mddefault}{\updefault}{$Y_n$}%
}}}} \put(3150,-1400){\makebox(0,0)[lb]{\smash{{\SetFigFont{10}{14.4}{\familydefault}{\mddefault}{\updefault}{$X_n$}%
}}}} \put(2330,-1300){\makebox(0,0)[lb]{\smash{{\SetFigFont{10}{14.4}{\familydefault}{\mddefault}{\updefault}{$Y_n$}%
}}}} \put(100,-3900){\makebox(0,0)[lb]{\smash{{\SetFigFont{10}{14.4}{\familydefault}{\mddefault}{\updefault}{$\widetilde{\rho}$}%
}}}} \put(100,-1700){\makebox(0,0)[lb]{\smash{{\SetFigFont{10}{14.4}{\familydefault}{\mddefault}{\updefault}{$\tau_2=\epsilon_n$}%
}}}} \put(-600,-2200){\makebox(0,0)[lb]{\smash{{\SetFigFont{10}{14.4}{\familydefault}{\mddefault}{\updefault}{$\widetilde{\sigma}_{x,n}(a_n)$}%
}}}} \put(-600,-2500){\makebox(0,0)[lb]{\smash{{\SetFigFont{10}{14.4}{\familydefault}{\mddefault}{\updefault}{$\widetilde{\sigma}_{x,n}$}%
}}}} \put(-1500,-3000){\makebox(0,0)[lb]{\smash{{\SetFigFont{10}{14.4}{\familydefault}{\mddefault}{\updefault}{$\widetilde{\sigma}_{x,n}(0)=\widetilde{\gamma}_{x,n}'(T_n)$}%
}}}} \put(100,-3509){\makebox(0,0)[lb]{\smash{{\SetFigFont{10}{14.4}{\familydefault}{\mddefault}{\updefault}{$\widetilde{\gamma}_{y,n}$}%
}}}} \put(5000,-3509){\makebox(0,0)[lb]{\smash{{\SetFigFont{10}{14.4}{\familydefault}{\mddefault}{\updefault}{$\widetilde{\gamma}_{x,n}$}%
}}}} \put(400,-2050){\makebox(0,0)[lb]{\smash{{\SetFigFont{10}{14.4}{\familydefault}{\mddefault}{\updefault}{$\gamma_{\tilde{\sigma}_{x,n}(a_n)}$}%
}}}} \put(1600,-3100){\makebox(0,0)[lb]{\smash{{\SetFigFont{10}{14.4}{\familydefault}{\mddefault}{\updefault}{$\widetilde{\beta}_n$}%
}}}} \put(2450,-3900){\makebox(0,0)[lb]{\smash{{\SetFigFont{10}{14.4}{\familydefault}{\mddefault}{\updefault}{$\tau_1=\tau_{1,n}$}%
}}}} \put(5500,-2099){\makebox(0,0)[lb]{\smash{{\SetFigFont{10}{14.4}{\familydefault}{\mddefault}{\updefault}{$\widetilde{\sigma}_{y,n}(b_n)$}%
}}}} \put(5800,-2500){\makebox(0,0)[lb]{\smash{{\SetFigFont{10}{14.4}{\familydefault}{\mddefault}{\updefault}{$\widetilde{\sigma}_{y,n}$}%
}}}} \put(5000,-3200){\makebox(0,0)[lb]{\smash{{\SetFigFont{10}{14.4}{\familydefault}{\mddefault}{\updefault}{$\widetilde{\sigma}_{y,n}(0)=\widetilde{\gamma}_{y,n}'(T_n)$}%
}}}} \put(3300,-3164){\makebox(0,0)[lb]{\smash{{\SetFigFont{10}{14.4}{\familydefault}{\mddefault}{\updefault}{$X_n$}%
}}}} \put(3300,-3654){\makebox(0,0)[lb]{\smash{{\SetFigFont{10}{14.4}{\familydefault}{\mddefault}{\updefault}{$X_n$}%
}}}} \put(2050,-3500){\makebox(0,0)[lb]{\smash{{\SetFigFont{10}{14.4}{\familydefault}{\mddefault}{\updefault}{$Y_n$}%
}}}} \end{picture}%
\caption{Foliations of the regions ${\mathcal{R}}_{x,n}$ and ${\mathcal{R}}_{y,n}$}  \label{figure:3} \end{center} \end{figure}

Since $\sigma_{x}$ and $\sigma_{y}$ and their images under $\varphi^{t},$
for $t\ge R,$ are approximately unstable, the curvatures of $\widetilde{\sigma}_{x,n}$
and $\widetilde{\sigma}_{y,n}$ are positive. For $t\ge T_{n},$ let
$E_{x,n}(t)$ be a unit normal field along $\widetilde{\gamma}_{x,n}(t)$
chosen so that $E_{x,n}(T_{n})$ is in the same direction as the derivative
at $s=0$ of the curve of basepoints of $\widetilde{\sigma}_{x,n}(s).$
Let $J_{x,n}(t)=j_{x,n}(t)E_{x,n}(t)$ be the Jacobi field $J_{x,n}(t)=(d/ds)|_{s=0}\gamma_{\tilde{\sigma}_{x,n}(s)}(t-T_{n}),$
for $t\ge T_{n}.$ Since $j_{x,n}'(T_{n})/j_{x,n}(T_{n})$ is equal
to the curvature at $s=0$ of the curve $s\mapsto\gamma_{\tilde{\sigma}_{x,n}(s)}(0)$
(as explained in Section \ref{sec:Cone Fields}) and $j_{x,n}(T_{n})>0,$
we have $j_{n,x}'(T_{n})>0.$ Therefore, if we extend the geodesic
$\widetilde{\gamma}_{x,n}|[T_{n},\infty)$ for $t<T_{n}$ up to the
point where $\widetilde{\gamma}_{x,n}$ exits $\widetilde{F}(\epsilon_{0})$
and let tr$(\widetilde{\gamma}_{x,n})$ denote the trace of this extended
geodesic, then for $a_{n}$ sufficiently small and $s\in(0,a_{n}],$
the distance between $\gamma_{\tilde{\sigma}_{x,n}(s)}(t)$ and tr$(\gamma_{\tilde{\sigma}_{x,n}(0)})$
is an increasing function of $t$ near $t=0.$ By the convexity of
this function, it follows that $\gamma_{\tilde{\sigma}_{x,n}(s)}(t)$
must leave $\widetilde{F}(\epsilon_{n})$ at some $t_{x,n,s}>0.$

The geodesic $\widetilde{\gamma}_{x,n}(t),$ $t\ge T_{n},$ and the
geodesics $\gamma_{\tilde{\sigma}_{x,n}(s)}(t),$ $0\le t\le t_{x,n,s},$
$0<s\le a_{n},$ form a foliation of a region $\mathcal{R}_{x,n}$
in $\widetilde{F}(\epsilon_{n})$ bounded on four sides by $\widetilde{\gamma}_{x,n}(t),$
$t\ge T_{n};$ the curve of basepoints of $\widetilde{\sigma}_{x,n}(s),$
$0\le s\le a_{n};$ $\gamma_{\tilde{\sigma}_{x,n}(a_{n})}(t),$ $0\le t\le t_{x,n,a_{n}};$
and $\{(\tau_{1},\tau_{2})\in\widetilde{F}(\epsilon_{n}):\tau_{1}\ge\tau_{x,n},\tau_{2}=\epsilon_{n}\},$
where $\tau_{x,n}$ is the $\tau_{1}$-coordinate of $\gamma_{\tilde{\sigma}_{x,n}(a_{n})}$
at the time it leaves $\widetilde{F}(\epsilon_{n}).$ The geodesics
in this foliation cannot intersect each other, because the curvature
is negative in $\widetilde{F}(\epsilon_{n}),$ which implies that
there are no focal points in $\widetilde{F}(\epsilon_{n}).$ Let $X_{n}$
be the unit vector field on $\mathcal{R}_{x,n}$ consisting of the
tangent vectors to the geodesics forming this foliation. There exists
$\overline{\tau}_{x,n}$ such that the region $\mathcal{R}_{x,n}$
contains all points in $\widetilde{F}(\epsilon_{n})$ with $\tau_{1}\ge\overline{\tau}_{x,n}$
that lie above tr$(\widetilde{\gamma}_{x,n}).$ Similarly, we may
construct a vector field $Y_{n}$ on a region $\mathcal{R}_{y,n}$
in $\widetilde{F}(\epsilon_{n})$ foliated by the geodesic $\widetilde{\gamma}_{y,n}(t),$
$t\ge T_{n},$ and the geodesics $\gamma_{\tilde{\sigma}_{y,n}(s)}(t),$
$0\le t<t_{y,n,s},$ $0<s\le b_{n},$ where $t_{y,n,s}$ is the time
at which $\gamma_{\tilde{\sigma}_{y,n}(s)}$ exits $\widetilde{F}(\epsilon_{n}).$
There exists $\overline{\tau}_{y,n}$ such that the region $\mathcal{R}_{y,n}$
contains all points in $\widetilde{F}(\epsilon_{n})$ with $\tau_{1}\le\overline{\tau}_{y,n}$
that lie above tr$(\widetilde{\gamma}_{y,n}).$ Let $T_{x,n}>T_{n}$
and $T_{y,n}>T_{n}$ be times at which $\tau_{1}(\widetilde{\gamma}_{x,n}(T_{x,n}))\ge\overline{\tau}_{x,n}$
and $\tau_{1}(\widetilde{\gamma}_{y,n}(T_{y,n}))\ge\overline{\tau}_{y,n}.$

Consider the projections $\widehat{\gamma}_{x,n},$ $\widehat{\gamma}_{y,n}$
of $\widetilde{\gamma}_{x,n},$ $\widetilde{\gamma}_{y,n},$ respectively,
to $\widehat{F}(\epsilon_{n}).$ Since $\widehat{\gamma}_{x,n}$ and
$\widehat{\gamma}_{y,n}$ approach $\widehat{\rho}$ from the same
side $(\tau_{2}>0),$ but with opposite orientations, they intersect
infinitely often. Thus there exist $t_{x,n}>T_{x,n}$ and $t_{y,n}>T_{y,n}$
such that $\widehat{\gamma}_{x,n}(t_{x,n})=\widehat{\gamma}_{y,n}(t_{y,n}).$
We now impose the additional condition that the lifts $\widetilde{\sigma}_{x,n}$
and $\widetilde{\sigma}_{y,n}$ be chosen so that $\widetilde{\gamma}_{x,n}(t_{x,n})=\widetilde{\gamma}_{y,n}(t_{y,n}).$

Let $(\tau_{1,n},\tau_{2,n})$ be the $(\tau_{1},\tau_{2})$ coordinates
of $\widetilde{\gamma}_{x,n}(t_{x,n}).$ At $(\tau_{1,n},\tau_{2,n})$
both $X_{n}$ and $Y_{n}$ have negative components in the $\partial/\partial\tau_{2}$
direction, and $X_{n}$ has a positive component in the $\partial/\partial\tau_{1}$
direction, while $Y_{n}$ has a negative component in the $\partial/\partial\tau_{1}$
direction. At $(\tau_{1,n},\epsilon_{n})$ both $X_{n}$ and $Y_{n}$
have positive components in the $\partial/\partial\tau_{2}$ direction,
and $X_{n}$ still has a positive component in the $\partial/\partial\tau_{1}$
direction, while $Y_{n}$ still has a negative component in the $\partial/\partial\tau_{1}$
direction. Therefore the Euclidean angle from $X_{n}$ to $Y_{n}$
measured in the counterclockwise direction changes from being greater
than $\pi$ to being less than $\pi$ along the vertical segment $\tau_{1}=\tau_{1,n},$
$\tau_{2,n}\le\tau_{2}\le\epsilon_{n}.$ By the intermediate value
theorem, there is a point $p_{n}$ along this segment such that the
angle from $X_{n}$ to $Y_{n}$ is $\pi.$ We let $\bar{a}_{n}\in(0,a_{n}),$
$\bar{b}_{n}\in(0,b_{n}),$ and $\bar{t}_{x,n},\bar{t}_{y,n}>T_{n}$
be such that $\gamma_{\tilde{\sigma}_{x}(\bar{a}_{n})}(\bar{t}_{x,n})=p_{n}=\gamma_{\tilde{\sigma}_{y}(\bar{b}_{n})}(\bar{t}_{y,n})$
and $\gamma'_{\tilde{\sigma}_{x,n}(\bar{a}_{n})}(\bar{t}_{x,n})=-\big(\gamma'_{\tilde{\sigma}_{y,n}(\bar{b}_{n})}(\bar{t}_{y,n})\big).$
We join $\gamma_{\tilde{\sigma}(\overline{a}_{n})}$ and $-\gamma_{\tilde{\sigma}(\overline{b}_{n})}$
at $p_{n}$ to form a geodesic $\widetilde{\beta}_{n}$ from the basepoint
of $\widetilde{\sigma}(\bar{a}_{n})$ to the basepoint of $\widetilde{\sigma}(\bar{b}_{n})$
in $\widetilde{F}(\epsilon_{n}),$ and we let $\beta_{n}$ be the
image of $\widetilde{\beta}_{n}$ under the projection from $\widetilde{F}(\epsilon_{n})$
to $F(\epsilon_{n}).$ The geodesic $\gamma_{n}:[0,L_{n}]\to S$ from
$x$ to $y$ is defined to be the concatenation of $\gamma_{\sigma_{x}(\bar{a}_{n})}|[0,T_{n}],$
$\beta_{n},$ and $-\big(\gamma_{\sigma_{y}(\bar{b}_{n})}|[0,T_{n}]\big).$
It follows from the construction that the second Fermi coordinate
of $\widetilde{\beta}_{n}$ is everywhere positive and less than $\epsilon_{n},$
as required. Moreover, if $1\le m<n,$ then $\widetilde{\beta}_{n}$
can be extended by joining it to lifts to $\widetilde{F}(\epsilon_{m})$
of $\gamma_{\sigma_{x}(\bar{a}_{n})}|[T_{m},T_{n}]$ and $-\big(\gamma_{\sigma_{y}(\bar{b}_{n})}|[T_{m},T_{n}]\big).$
For this extension of $\widetilde{\beta}_{n},$ the second Fermi coordinate
is everywhere positive and less than $\epsilon_{m}.$ This implies
(\ref{eq:nmBGproperty}).
\end{proof}
The following theorem is a special case of a theorem of S. \L ojasiewicz
\cite{Lojasiewicz}. (See also Theorem 4.4 in the expository article
\cite{BierstoneMilman}.)
\begin{thm}
\label{thm:Lojasiewicz}Suppose $M$ is a connected real analytic
surface, $K$ is a compact subset of $M,$ and $f:M\to\mathbb{R}$
is a real analytic function. Assume that $f$ does not vanish identically
on $M.$ Then there exists a finite set of points $P\subset M$ and
a set $A$ consisting of the union of finitely many real analytic
curves on $M$, where $P$ and/or $A$ may be empty, such that \[
\{y\in K:f(y)=0\}=K\cap(P\cup A).\]

\begin{prop}
\label{pro:AnalyticNonconcurrence}If $T>0,$ and $(x,z)\in S\times S,$
then there are at most finitely many unit speed geodesics from $x$
to $z$ of length less than or equal to $T.$
\end{prop}
\end{thm}
\begin{proof}
Suppose the lemma were false. Then there exists an infinite sequence
$\gamma_{n}:[0,T]\to S$ of unit speed geodesics with $\gamma_{n}(0)=x$
and $\gamma_{n}(t_{n})=z$ for some $t_{n}\in[0,T].$ By passing to
a subsequence of $(\gamma_{n})_{n=1,2,\dots}$ and reindexing, we
may assume that $\lim_{n\to\infty}t_{n}=t_{0}\in(0,T]$ and $\lim_{n\to\infty}\gamma_{n}'(0)=v_{0}\in T_{x}^{1}M.$
Let $f:T_{x}S\to\mathbb{R}$ be defined by $f(v)=({\rm dist}(\exp_{x}v,z))^{2}.$
Since $y\mapsto({\rm dist}(y,z))^{2}$ is a real analytic function
in a neighborhood of $z,$ there exists an open disk $M$ about $t_{0}v_{0}$
in $T_{x}S$ such that $f$ restricted to $M$ is real analytic. Let
$K$ be a closed disk about $t_{0}v_{0}$ that is contained in $M.$
Since $f$ vanishes on an infinite subset of $K,$ Theorem \ref{thm:Lojasiewicz}
implies that there exists a non-trivial real analytic curve $\alpha(s),$
$-\delta<s<\delta,$ in $T_{x}M$ such that $f(\alpha(s))=0$ for
all $s\in(-\delta,\delta).$ Consider the variation by (not necessarily
unit speed) geodesics, $s\to\exp(t\alpha(s))$ , where $0\le t\le1$.
These geodesics pass through $x$ when $t=0,$ and they pass through
$z$ when $t=1.$ By the first variation formula for arc length (see,
e.g., \cite{Chavel}), $(d/ds)(|\alpha(s)|)\equiv0.$ Thus there exists
$L>0$ such that $|\alpha(s)|=L$ for all $s\in(-\delta,\delta).$
This implies that $f$ vanishes along an arc of the circle $|v|=L.$
Therefore $f$ vanishes identically on this circle. Thus every unit
speed geodesic starting at $x$ passes through $z$ at time $L.$
We can repeat this argument at $z$ to conclude that every unit speed
geodesic starting at $z$ passes through $x$ at time $L.$ Hence
every unit speed geodesic starting at $x$ is at a point conjugate
to $x$ along that geodesic at times $t=L,2L,3L,\dots$. However,
by Proposition \ref{pro:Asymptotic}, there is a geodesic starting
at $x$ that eventually remains in the negative curvature region,
which implies there are at most finitely many points conjugate to
$x$ along this geodesic. This is a contradiction. \end{proof}
\begin{lem}
\label{lem:AvoidsOmega}Let $(x,y)\in S\times S$ and suppose $\gamma_{n}:[0,L_{n}]\to S$,
$n=1,2,\dots$, is an infinite family of distinct geodesics from $x$
to $y$ as in Proposition \ref{pro:BangertGutkinProperty}. If $\Omega$
is a finite subset of $S,$ then there exists an infinite subsequence
$(\gamma_{n_{k}})_{k=1,2,\dots}$ of $(\gamma_{n})_{n=1,2,\dots}$
such that $\gamma_{n_{k}}((0,L_{n_{k}}))\cap\Omega=\emptyset$ for
$k=1,2,\dots.$
\begin{proof}
Let $\Sigma\subset S$ be the set of self-intersection points of $\rho.$
If $\Sigma\ne\emptyset,$ then there exist $\epsilon_{1}>0$ and $C>1,$
depending on the angles made by $\rho$ at points in $\Sigma,$ such
that if $z$ is in the $\epsilon_{1}$-neighborhood of tr$(\rho),$
and $z$ has at least one representation in Fermi coordinates $(\tau_{1},\tau_{2})$
along $\rho$ with $\tau_{2}\ne0,$ then $z\notin\Sigma$ and \begin{equation}
0<\text{dist}(z,\Sigma)<C|\tau_{2}|.\label{eq:Sigma_distance}\end{equation}
Let $\alpha_{1}=\min\{\text{dist}(z,\text{tr}(\rho)):z\in\Omega\setminus\text{tr}(\rho)\}$
and $\alpha_{2}=\min\{\text{dist}(z,\Sigma):z\in(\Omega\cap\text{tr}(\rho))\setminus\Sigma\}.$
(We define $\min(\emptyset)=\infty.)$ Let $0<\epsilon<\min(\epsilon_{1},\alpha_{1},\alpha_{2}/C),$
and let $T=T(\epsilon)$ and $N=N(\epsilon)$ be as in Proposition
\ref{pro:BangertGutkinProperty}. For $n>N,$ there exists a smooth
choice of the coordinate $\tau_{2}$ along $\gamma_{n}|[T,L_{n}-T]$
such that\[
0<|\tau_{2}(\gamma_{n}(t))|<\epsilon,\]
for $t\in[T,L_{n}-T].$ If a point $\gamma_{n}(t),$ for some $n>N$
and some $t\in[T,L_{n}-T],$ is in tr$(\rho),$ then it is not in
$\Sigma,$ and it is closer to $\Sigma$ than any point in $(\Omega\cap\text{tr}(\rho))\setminus\Sigma.$
If it is not in tr$(\rho),$ then it is closer to tr$(\rho)$ than
any point in $\Omega\setminus\text{tr}(\rho).$ Therefore $\gamma_{n}([T,L_{n}-T])\cap\Omega=\emptyset,$
for $n>N.$ By applying Proposition \ref{pro:AnalyticNonconcurrence}
to points of the form $(x,z)$ or $(y,z),$ where $z\in\Omega,$ we
see that there exist infinitely many $n>N$ such that $\gamma_{n}\big((0,T)\cup(L_{n}-T,L_{n})\big)\cap\Omega=\emptyset.$
\end{proof}
\end{lem}
Below is the proof of our main result.

\begin{proof}[Proof of Theorem \ref{thm:GeneralSecurity}]Let $(x,y)\in S\times S,$
and let $\gamma_{n}:[0,L_{n}]\to S,$ $n=1,2,\dots,$ be an infinite
family of distinct geodesics from $x$ to $y$ as in Proposition \ref{pro:BangertGutkinProperty}.
By applying Lemma \ref{lem:AvoidsOmega} to $\Omega=\{x,y\}$ and
passing to a subsequence and reindexing, we may assume that the geodesics
$\gamma_{n}$ pass through $x$ and $y$ only at the endpoints.

We will prove inductively that there exists a strictly increasing
sequence of positive integers $n_{1},n_{2},\dots,$ such that

\begin{equation}
\text{no three of }\gamma_{n_{1}},\gamma_{n_{2}},\dots,\gamma_{n_{k}}\text{ are concurrent except at }x\text{ and at }y.\label{eq:nonconcurrent}\end{equation}
 We may take $n_{1}=1$ and $n_{2}=2.$ Then (\ref{eq:nonconcurrent})
is clearly satisfied for $k=1,2.$ Now assume $k\ge2$ and we have
found $n_{1}<n_{2}<\cdots<n_{k}$ such that (\ref{eq:nonconcurrent})
holds. We will show that we can choose $n_{k+1}>n_{k}$ such that
(\ref{eq:nonconcurrent}) holds with $k$ replaced by $k+1.$ We now
apply Lemma \ref{lem:AvoidsOmega} with \[
\Omega=\bigcup_{1\le i<j\le k}\left(\gamma_{n_{i}}((0,L_{n_{i}}))\cap\gamma_{n_{j}}((0,L_{n_{j}}))\right).\]
 Since the geodesics $\gamma_{n}$ pass through $x$ and $y$ only
at the endpoints, $\Omega$ is a finite set. Thus Lemma \ref{lem:AvoidsOmega}
implies that there exists $n_{k+1}>N$ such that $\gamma_{n_{k+1}}((0,L_{n_{k+1}}))\cap\Omega=\emptyset,$
and (\ref{eq:nonconcurrent}) holds with $k$ replaced by $k+1.$
Therefore there exists an infinite sequence of positive integers $n_{1},n_{2},\dots$
such that no three of $\gamma_{n_{1}},\gamma_{n_{2}},\dots$ are concurrent
except at $x$ and at $y.$ Since any point can be in at most two
of $\gamma_{n_{1}}((0,L_{n_{1}})),\gamma_{n_{2}}((0,L_{n_{2}})),\dots,$
it follows that there does not exist a finite blocking set for $(x,y).$

\end{proof}

\section{Acknowledgments}

We thank Eugene Gutkin for his encouragement and helpful correspondence,
and we thank Ji-Ping Sha for making a correction to an earlier version
of our proof of Proposition \ref{pro:AnalyticNonconcurrence}.

\end{document}